\documentclass[11pt,reqno]{amsart}
\usepackage{amssymb}
\usepackage{amsfonts}
\usepackage{amsthm}
\usepackage{amscd}
\usepackage{amssymb}

\theoremstyle{plain}
\newtheorem{theorem}{Theorem}[section]

\newtheorem{proposition}[theorem]{Proposition}

\theoremstyle{definition}
\newtheorem{definition}[theorem]{Definition}
\newtheorem{example}[theorem]{Example}

\theoremstyle{remark}
\newtheorem{remark}[theorem]{Remark}

\numberwithin{equation}{section}
\newcommand{\refE}[1]{(\ref{E:#1})}
\newcommand{\refS}[1]{Section~\ref{S:#1}}
\newcommand{\refSS}[1]{Section~\ref{SS:#1}}
\newcommand{\refT}[1]{Theorem~\ref{T:#1}}

\newcommand{\refP}[1]{Proposition~\ref{P:#1}}

\newcommand{\refR}[1]{Remark~\ref{R:#1}}
\newcommand{\R}{\ensuremath{\mathbb{R}}}

\newcommand{\C}{\ensuremath{\mathbb{C}}}
\newcommand{\N}{\ensuremath{\mathbb{N}}}

\newcommand{\Pro}{\ensuremath{\mathbb{P}}}

\newcommand{\F}{\ensuremath{\mathbb{F}}}

\newcommand{\A}{\ensuremath{\mathcal{A}}}

\newcommand{\iu}{{\,\mathrm{i}\,}}

\newcommand{\w}{\omega}
\newcommand{\cim}{C^\infty(M)}
\newcommand{\Cim}{C^\infty(M)}
\newcommand{\dpa}{\partial}
\newcommand{\db}{\overline{\partial}}
\newcommand{\hh}{\hat h}
\newcommand{\Ho}{\mathrm{H}}
\newcommand{\e}{\mathrm{e}}

 \newcommand{\ghm}[1][m]{\Gamma_{hol}(M,L^{#1})}

 \newcommand{\ghmb}[1][m]{\Gamma_{hol}^b(M,L^{#1})}

 \newcommand{\gulm}[1][m]{\Gamma_{\infty}(M,L^{#1})}
 
 \newcommand{\wb}{\overline{w}}
 \newcommand{\zb}{\overline{z}}

  \newcommand{\pnc}[1][N]{\ensuremath{\Pro^{#1}(\C)}}

 \newcommand{\skp}[2]{{\langle #1,#2\rangle}}
 \newcommand{\skps}[3]{{\langle #1,#2\rangle}_{#3}}
 \newcommand{\skpsm}[3]{{\langle #1,#2\rangle}^{(m)}_{#3}}

\newcommand{\Lpm}{{\mathrm{L}}^2(M,L^m)}

\newcommand{\Lqv}{{\mathrm{L}}^2(Q,\mu)}

\newcommand{\Tfm}{T_f^{(m)}}
\newcommand{\Tgm}{T_g^{(m)}}
\newcommand{\Tfgm}{T_{\{f,g\}}^{(m)}}
\newcommand{\Tma}[1]{T_{#1}^{(m)}}
\newcommand{\Hm}{\mathcal {H}^{(m)}}
\newcommand{\Hc}{{\mathcal {H}}}
\newcommand{\End}{\mathrm{End}}
\newcommand{\tr}{\operatorname{tr}}
\newcommand{\Tr}{\operatorname{Tr}}
\newcommand{\volu}{\operatorname{vol}}
\newcommand{\Imb}{I^{(m)}}
\newcommand{\Om}{\Omega}
\newcommand{\Ome}{\Omega^{(m)}_\epsilon}

\newcommand{\eghm}{\mathrm{End}(\Gamma_{hol}(M,L^{(m)}))}
\newcommand{\pimh}{\hat\Pi^{(m)}}
\newcommand{\Bm}{\mathcal{B}_m}
\newcommand{\Bh}{\mathcal{B}_{\hbar}}
\newcommand{\Pim}{\Pi^{(m)}}

\newcommand{\la}{\lambda}
\newcommand{\al}{\alpha}

\newcommand{\eam}{e_{\alpha}^{(m)}}
\newcommand{\ebm}{e_{\beta}^{(m)}}

\newcommand{\sm}{\sigma^{(m)}}
\newcommand{\svm}{{\check\sigma}^{(m)}}
\newcommand{\epsm}{\epsilon^{(m)}}

\begin{document}

\title[Berezin-Toeplitz quantization]{Berezin-Toeplitz Quantization\\
  and
Star Products\\ for Compact K\"ahler Manifolds}


\author[Martin Schlichenmaier]{Martin Schlichenmaier}
\address{%
University of Luxembourg\\ 
Mathematics Research Unit, FSTC\\
Campus Kirchberg\\ 6, rue Coudenhove-Kalergi,
L-1359 Luxembourg-Kirchberg\\ Luxembourg
}
\curraddr{}
\email{martin.schlichenmaier@uni.lu}
\thanks{Partial Support by the                           
Internal Research Project  GEOMQ08,  University of Luxembourg,       
is acknowledged.}


\date{13.2.2012, rev. 25.5.2012}

\begin{abstract}
For compact quantizable K\"ahler manifolds certain naturally defined 
star products and their constructions are reviewed. The presentation
centers around the Berezin-Toeplitz quantization scheme which is 
explained. As star products the Berezin-Toeplitz, Berezin,
and star product of geometric quantization
are  treated in detail. 
It is shown that all three are
equivalent. A prominent role is played by the Berezin
transform and its asymptotic expansion. 
A few ideas on two general constructions of star products 
of separation of variables type  by  
Karabegov and by  Bordemann--Waldmann respectively are given.
Some of the results presented is
work of the author partly joint with Martin Bordemann, Eckhard
Meinrenken and Alexander Karabegov.
At the end some works which make use of graphs in the 
construction and calculation of these star products
are sketched.
\end{abstract}

\maketitle


\section{Introduction}\label{S:intro}
Without any doubts the concepts of quantization is of fundamental
importance in modern physics. These concepts are equally 
influential in mathematics. 
The problems appearing in the physical treatments 
give a whole variety of
questions to be solved by mathematicians. Even more,
quantization
challenges mathematicians to develop corresponding mathematical concepts
with  necessary rigor.
Not only that they are inspiring in the sense that we mathematician provide 
solutions, but these developments will help to advance our
mathematical disciplines. 
It is not the place here to try to give some precise definition
what is quantization. I only mention that one mathematical aspect
of quantization is to pass from the classical ``commutative'' world
to the quantum ``non-commutative'' world.
There are many possible aspects of this passage.
One way is to replace the algebra of classical physical
observables (functions depending  locally on ``position'' and 
``momenta''), i.e. the commutative algebra of functions on the
phase-space manifold, by a non-commutative algebra of operators
acting  on a certain Hilbert space.
Another way is to ``deform'' the pointwise product in the 
algebra of functions into some non-commutative product $\star$.
The first method is called operator quantization, the second 
deformation quantization and the product $\star$ is called a star
product.
In both cases  
by some limiting process
the classical 
situation should be recovered.
I did not touch the question whether it is possible at all to obtain
such objects if one poses certain desirable conditions.
For example, the desired properties for the star product
(to be explained in the article further down) does not allow to deform
the product inside of the function algebra for all functions.
One is forced to pass to the algebra of formal power series
over the functions and deform there.
The resulting object will be a formal deformation quantization.

A special case of the operator method is geometric quantization.
One chooses a complex hermitian (pre)quantum line bundle 
on the phase space manifold. The operators 
act on the space of global sections of the bundle or on suitable subspaces. 
In the  that we can endow our phase-space manifold with the
structure of a K\"ahler manifold 
(and  only this case we are considering here)
we have a more rigid situation.
Our quantum line bundle should  carry a holomorphic  structure, if
the bundle exists at all.
The passage to the classical limit will be obtained by considering
higher and higher tensor powers of the quantum line bundle.
The sections of the bundle are the candidates of the quantum states.
But they depend on too many independent variables. 
In the 
K\"ahler setting there is the naturally defined subspace
of holomorphic sections. These sections are constant in anti-holomorphic
directions. They will be the quantum states.
This selection is sometimes  called  K\"ahler 
polarization.

In this review we will mainly deal with another type of operators 
on the space of holomorphic sections of the bundle. 
These  will be the Toeplitz operators. They are
naturally defined for K\"ahler manifolds. 
The assignment 
defines  the
Berezin-Toeplitz (BT) quantization
scheme. Berezin himself considered it for
certain special manifold \cite{BerCCS}, \cite{Bergc}.

Being a quantum line bundle means that the
curvature of the holomorphic hermitian line bundle is essentially equal to
the K\"ahler form. See \refS{setup} for the precise formulation.
A K\"ahler manifold is called quantizable if it admits a quantum line
bundle.
We will explain below that this is really a condition which 
not always
can be fulfilled.

The author in joint work with Martin Bordemann and Eckhard Meinrenken
\cite{BMS}
showed that at least in the compact quantizable K\"ahler case the
BT-quantization has the correct semi-classical limit behavior, hence
it is a quantization, see \refT{bms}.
In the compact K\"ahler case the operator of geometric quantization
is asymptotically related to the   
Toeplitz operator, see \refE{tuyn}.
The details are presented in \refS{btq}.

The special feature of the Berezin-Toeplitz quantization approach is that
it does not only provide an operator quantization but 
also an intimately
related star product, the Berezin-Toeplitz star product $\star_{BT}$.
It is obtained by ``asymptotic expansion'' of the product of
the two Toeplitz operators associated to the two functions to 
be $\star$-multiplied, see \refE{starapp}.
After recalling the definition of a star product in \refSS{star},
the
results about existence and the properties of $\star_{BT}$ are
given in \refSS{btq}. These  are results of the author partly
in
joint work with Bordemann, Meinrenken, and Karabegov.
The star product is a star product of separation of variables type
(in the sense of Karabegov) or equivalently of Wick type (in the
sense of Bordemann and Waldmann). We recall Karabegov's construction
of star products of this type. In particular, we discuss his formal
Berezin transform. 

In \refS{globToe} we introduce the disc bundle associated to the
quantum line bundle and introduce the global Toeplitz operators.
The individual Toeplitz operators for each tensor power 
of the line bundle correspond
to its  modes. The symbol calculus of generalized Toeplitz operators
due to Boutet de Monvel and Guillemin \cite{BGTo} is used to 
prove some parts of the above mentioned results.
In \refSS{const} as an illustration we explain how
$\star_{BT}$ is constructed. 

Other important techniques which we use in this context are
Berezin-Rawnsley's coherent states, co- and contra-variant symbols
\cite{CGR1} 
\cite{CGR2} 
\cite{CGR3} 
\cite{CGR4}. 
Starting from a function on $M$, assigning to it its Toeplitz
operator
and then calculating the covariant symbol of the operator will yield
another function. The corresponding map on the space of function is
called Berezin transform $I$, see \refS{bere}.
The map will depend on the chosen tensor power $m$ of the line bundle.
\refT{btrans}, obtained jointly with Karabegov, shows that it has
a complete asymptotic expansion.
One of the ingredients of the proof is the off-diagonal expansion 
of the Bergman kernel in the neighborhood of the diagonal \cite{KS}.

With the help of the Berezin transform $I$ the Berezin star product
can
be defined
\begin{equation*}
f\star_B g:=I(I^{-1}(f)\star_{BT}I^{-1}(g)).
\end{equation*}
In Karabegov's terminology both star products 
are dual and opposite to each other.

In \refSS{summary} a summary of the naturally defined star products
are given. These are 
$\star_{BT}$, 
$\star_{B}$,
 $\star_{GQ}$ (the star product of geometric quantization),
$\star_{BW}$ (the star product of Bordemann and Waldmann constructed
in  a manner \`a la Fedosov, see \refSS{bw}).
The star products 
$\star_{BT},\star_{BW}$ are of separation of variables type,
$\star_{B}$ also but with the role of holomorphic and antiholomorphic 
variables switched,
 $\star_{GQ}$ is neither nor.
The first three star products are equivalent.

How the knowledge of the asymptotic expansion of the Berezin transform
will allow to calculate the coefficients of the 
Berezin star product and recursively of the Berezin-Toeplitz star
product
is explained in \refSS{calc}.

In the \refS{graph} we 
consider the Bordemann-Waldmann star product \cite{BW} and 
make some remarks how graphs are 
of help in expressing the star product in a convenient form.
The work of Reshetikhin and Takhtajan \cite{ResTak}, 
Gammelgaard \cite{Gam}, and Huo Xu \cite{Huo1}, \cite{Huo2} 
are sketched.

In an excursion  we describe Kontsevich's construction 
\cite{Kont} of a
star product for arbitrary Poisson structures on $\R^n$.

The closing \refS{appl} gives 
some applications of the Berezin-Toeplitz
quantization
scheme.

This review is based on a talk which I gave in the frame of
the {\it Thematic Program on Quantization, Spring 2011},
at the University of Notre Dame, USA.
Some of the material was added on the basis of the questions
and the discussions 
of the audience. I am grateful to the organizers 
Sam Evens, Michael Gekhtman, Brian Hall, and Xiaobo Liu,
and 
to the audience. 
All of them made
this activity such a pleasant and successful event.
In its present version the review supplements and updates 
\cite{SchlBer},\cite{Schlgeoquant}.
Other properties, like the properties of the coherent state embedding,
more about Berezin symbols, traces 
and examples can be found there.
In particular, \cite{SchlBer} contains a more complete list
of related works of other authors.


\section{The geometric setup}\label{S:setup}

In the following let $\ (M,\w)$ be a
K\"ahler manifold. 
This means $M$ is a complex manifold (of complex dimension $n$)
and $\w$, the K\"ahler form, is
a non-degenerate closed positive $(1,1)$-form.
In the interpretation of physics $M$ will be the phase-space manifold.
(But besides the jargon we will use nothing from physics here.)
Further down we will assume that $M$ is compact.

\medskip

Denote by $\Cim$ the algebra of 
complex-valued (arbitrary often) differentiable functions
with  associative product given by
point-wise multiplication.
After forgetting the complex structure of $M$, our form $\omega$
will become a  symplectic form
and we introduce on 
$\Cim$ a Lie algebra structure, the 
{\em Poisson bracket} $\{.,.\}$, in the following way.
First we 
assign to every  $f\in\Cim$ its {\em Hamiltonian vector field} $X_f$, and 
then to every pair of functions $f$ and $g$ the 
{\em Poisson bracket} $\{.,.\}$ via
\begin{equation}
\label{E:Poi}
\w(X_f,\cdot)=df(\cdot),\qquad  
\{\,f,g\,\}:=\w(X_f,X_g)\ .
\end{equation}
In this way 
$\Cim$ becomes a 
{\em Poisson algebra}, i.e. we have the compatibility
\begin{equation}
\{h,f\cdot g\}=\{h,f\}\cdot g+f\cdot \{h,g\},\qquad f,g,h\in\Cim.
\end{equation}

\medskip

The next step in the geometric set-up  is the choice of
a quantum line bundle.
In the K\"ahler case 
a {\em quantum line bundle} for  $(M,\w)$ is 
a triple $(L,h,\nabla)$, where $L$ is a holomorphic line bundle, 
$h$ a Hermitian metric   on $L$, and
 $\nabla$ a connection  compatible with the metric $h$
and the complex structure, 
such that the (pre)quantum condition 
\begin{equation}\label{E:quant}
\begin{gathered}
\mathrm{curv}_{L,\nabla}(X,Y):=\nabla_X\nabla_Y-\nabla_Y\nabla_X-\nabla_{[X,Y]}
=
-\iu\w(X,Y),\\
\text{in other words} \quad \mathrm{curv}_{L,\nabla}=-\iu\w \ 
\end{gathered}
\end{equation}
is fulfilled.
By the compatibility requirement $\nabla$ is uniquely fixed.
With respect to a local holomorphic frame of the bundle
the metric $h$ will be  
represented by 
a function $\hh$.
Then  the curvature with respect to the compatible
connection 
is given 
by $\db\dpa\log\hh$. 
Hence,  
the quantum condition reads as  
\begin{equation}\iu\db\dpa\log\hh=\w\ .
\end{equation}
If there exists such a quantum line bundle for $(M,\omega)$ then
$M$ is called quantizable. Sometimes  the pair manifold and quantum line
bundle is called quantized K\"ahler manifold.

\noindent
\begin{remark}\label{R:ample} 
Not all K\"ahler manifolds are quantizable.
In the compact K\"ahler case from \refE{quant}  it
follows that the curvature is a positive form, hence 
$L$ is a positive line bundle.
By the Kodaira embedding theorem \cite{SchlRS}
there exists a positive
tensor power $L^{\otimes m_0}$ which has enough global holomorphic
sections to embed the complex manifold $M$ via these sections
into  projective space $\Pro^N(\C)$ of suitable dimension $N$.
By Chow's theorem \cite{SchlRS} it is a smooth projective
variety.
The line bundle  $L^{\otimes m_0}$ which gives an embedding 
is called very ample.
This implies for 
example, that only those higher dimensional  
complex tori are  quantizable which admit  ``enough 
theta functions'', i.e.  which are {abelian varieties}.

A warning is in order, let $\phi:M\mapsto \Pro^N(\C)$ be the
above mentioned 
embedding as complex manifolds. This embedding is in general not
a K\"ahler embedding, i.e. $\phi^*(\w_{FS})\ne \w$, where
$\w_{FS}$ is the standard Fubini-Study K\"ahler form for
$\Pro^N(\C)$. Hence, we cannot restrict our attention only
on K\"ahler submanifolds of projective space.
\end{remark}

For compact K\"ahler manifolds 
we will always assume that 
the quantum bundle $L$ itself is already
very ample.
This 
is not a restriction as   $L^{\otimes m_0}$ will be a
quantum line bundle for the
rescaled 
K\"ahler form $m_0\w$ for the same complex manifold $M$.

Next, we consider all  positive tensor powers 
of the quantum line bundle:
\newline
$(L^m,h^{(m)},\nabla^{(m)})$,
here  $L^m:=L^{\otimes m}$ and $h^{(m)}$ and $\nabla^{(m)}$ are
  naturally
extended.
We 
introduce a product on the space of sections. 
First we take the Liouville form
$\ \Omega=\frac 1{n!}\w^{\wedge n}\ $ as volume form  on $M$
and then set 
for  the  product
and the norm
on  the space $\gulm$ of global $C^\infty$-sections
 (if they are finite)
\begin{equation}
\label{E:skp}
\langle\varphi,\psi\rangle:=\int_M h^{(m)} (\varphi,\psi)\;\Omega\  ,
\qquad
||\varphi||:=\sqrt{\langle \varphi,\varphi\rangle}\ .
\end{equation}
Let $\Lpm$ be the  L${}^2$-completed space of 
bounded sections with respect to this  norm. 
Furthermore, let
$\ghmb$ be 
the space of 
global holomorphic
sections of $L^m$  which are bounded.
It can be identified with a closed subspace of 
$\Lpm$. Denote by 
\begin{equation}
\ \Pim:\Lpm\to\ghmb\ 
\end{equation}
the orthogonal projection.

If the manifold $M$ is compact 
``being bounded'' is of course no restriction. Furthermore,
$\ghm=\ghmb$ and this space 
is finite-dimensional.
Its dimension $N(m):=\dim\ghm$ will be given by the 
Hirzebruch-Riemann-Roch Theorem \cite{SchlRS}.
 Our projection will be
\begin{equation}\label{E:proj}
\ \Pim:\Lpm\to\ghm\ .
\end{equation}
If we fix an orthonormal basis $s_l^{(m)}, l=1,\dots,N(m)$
of $\ghm$
then%
\footnote{In my convention the scalar product is anti-linear in the
first argument.}
\begin{equation}\label{E:proje}
\Pim(\psi)=\sum_{l=1}^{N(m)}\skp {s_l^{(m)}}{\psi}\cdot s_l^{(m)}.
\end{equation}

\section{Berezin-Toeplitz operator quantization}\label{S:btq}

Let us start with the compact K\"ahler manifold case. I will make
some remarks at the end of this section on the general setting.
In the interpretation of physics, our manifold $M$ is a phase-space.
Classical observables are (real-valued) functions on the phase space.
Their values are the physical values to be found by experiments.
The classical observables  commute under  pointwise multiplication.
One of the aspects of quantization is
to replace the classical observable
by something which is non-commutative. One approach is to 
replace the functions  by operators on a certain Hilbert space 
(and the physical values to be measured should
correspond to eigenvalues of them).
In the Berezin-Toeplitz (BT) operator
quantization this is done as follows.
\begin{definition}
For a function $f\in\Cim$ 
the associated {\it Toeplitz operator  $\Tfm$ (of level $m$)}
is defined as
\begin{equation}\label{E:toeplitz}
\Tfm:=\Pim\, (f\cdot):\quad\ghm\to\ghm\ .
\end{equation}
\end{definition}
\noindent
In words: One takes a holomorphic section $s$ 
and multiplies it with the
differentiable function $f$. 
The resulting section $f\cdot s$ will only be  differentiable.
To obtain a holomorphic section, one has to project  it back
on the subspace of holomorphic sections.

With respect to the explicit representation 
\refE{proje} we obtain
\begin{equation}
\Tfm(s):=
\sum_{l=1}^{N(m)}
\skp {s_l^{(m)}}{f\,s}\; s_l^{(m)}.
\end{equation}
After expressing the  scalar product \refE{skp} we get a 
representation of $\Tfm$ as an integral
\begin{equation}
\Tfm(s)(x)=\int_M  f(y)\,
\left(\sum_{l=1}^{N(m)}
h^{(m)}({s_l^{(m)}},{s})(y)\, s_l^{(m)}(x)\right)\,\Omega(y).
\end{equation}
The space $\ghm$ is the {\it quantum space (of level $m$)}.
The linear map
\begin{equation}
\Tma {}:\Cim\to \End\big(\ghm\big),\quad  f\to \Tma f=\Pi^{(m)}(f\cdot)
\ , m\in\N_0
\end{equation}
is the  {\it Toeplitz}  or {
\it Berezin-Toeplitz quantization map (of level $m$)}.
It will 
neither be a Lie algebra homomorphism nor
an associative algebra homomorphism as
in general
$$
T^{(m)}_f\, T^{(m)}_g=\Pi^{(m)}\,(f\cdot)\,\Pi^{(m)}\,(g\cdot)\,\Pi^{(m)}\ne
\Pi^{(m)}\,(fg\cdot)\,\Pi =T^{(m)}_{fg}.
$$
For $M$  a compact K\"ahler manifold
it was already mentioned that 
the space $\ghm$ is finite-dimensional.
On a fixed  level $m$ the BT quantization 
is a map
from the infinite dimensional 
commutative algebra of functions to a noncommutative
finite-dimensional (matrix) algebra.
A lot of classical information will get lost.
 To recover this
information one has to  consider not just a single level $m$ but
all levels together as done in the 
\begin{definition}\label{D:BT}
The Berezin-Toeplitz (BT) quantization is the map
\begin{equation}\label{E:BT}
\Cim\to\prod_{m\in\N_0}\eghm,\qquad
f\to(\Tfm)_{m\in\N_0}.
\end{equation}
\end{definition}
In this way 
a family of
finite-dimensional (matrix) algebras and a family of maps are obtained,
which  in the classical limit should ``converges'' to the
algebra $\Cim$.
That this is indeed the case and what ``convergency'' means will be
made precise in the following.

Set for $f\in\Cim$ by $|f|_\infty$ the sup-norm of $\ f\ $ on $M$ and by 
\begin{equation}
||\Tfm||:=\sup_{\substack {s\in\ghm\\ s\ne 0}}\frac {||\Tfm s||}{||s||}
\end{equation}
the operator norm with respect to the norm \refE{skp}
on $\ghm$.

That the BT quantization is indeed a quantization in the sense that
it has the correct semi-classical limit,
or that it is a strict quantization in the sense of Rieffel,
 is the content of the
following theorem from 1994.

\begin{theorem}
\label{T:bms}
[Bordemann, Meinrenken, Schlichenmaier]
\cite{BMS}

\noindent
(a) For every  $\ f\in \Cim\ $ there exists a $C>0$ such that   
\begin{equation}\label{E:norma}
|f|_\infty -\frac Cm\quad
\le\quad||\Tfm||\quad\le\quad |f|_\infty\ .
\end{equation}
In particular, $\lim_{m\to\infty}||\Tfm||= |f|_\infty$.

\noindent
(b) For every  $f,g\in \Cim\ $ 
\begin{equation}
\label{E:dirac}
||m\iu[\Tfm,\Tgm]-\Tfgm||\quad=\quad O(\frac 1m)\ .
\end{equation}

\noindent
(c) For every  $f,g\in \Cim\ $ 
\begin{equation}\label{E:prod}
||\Tfm\Tgm-T^{(m)}_{f\cdot g}||\quad=\quad O(\frac 1m)
\ .
\end{equation}
\end{theorem}
\noindent
The original proof uses the machinery of generalized Toeplitz
structures and operators as developed by Boutet de Monvel
and Guillemin \cite{BGTo}.
We will give a sketch of some parts of 
the proof in \refS{globToe} and \refSS{norm}.
In the meantime 
there also exists other proofs on the basis of 
Toeplitz kernels, Bergman kernels, Berezin transform etc.
Each of them give very useful additional insights.

We will need in the following from 
\cite{BMS}
\begin{proposition}\label{P:tsur}
On every level $m$  the Toeplitz map  
$$
\Cim\to\eghm,\qquad  f\to\Tfm,
$$
is surjective.
\end{proposition}
Let us mention that for real-valued $f$ the Toeplitz operator
$\Tfm$ will be selfadjoint. Hence, they have real-valued eigenvalues.
\begin{remark}
(Geometric Quantization.)
 Kostant and Souriau introduced  
the operators of geometric quantization
in this geometric setting.
In a first step the prequantum operator associated to the
bundle $L^m$ (and acting on its sections) for the function $f\in\Cim$ 
is defined as
$P_f^{(m)}:=\nabla_{X_f^{(m)}}^{(m)}+\iu f\cdot  id$.
Here 
$X_f^{(m)}$ is the Hamiltonian vector field of $f$ with respect to the
K\"ahler form $\w^{(m)}=m\cdot \w$
and $\nabla_{X_f^{(m)}}^{(m)}$ is the covariant derivative.
In the context of geometric quantization 
one has to choose a polarization.
This corresponds to the fact that the ``quantum states'', i.e. 
the sections of the quantum line bundle, should only depend on 
``half of the variables'' of the phase-space manifold $M$.
In general, such a polarization  will not be 
unique. 
But in our complex situation there is a 
canonical one by taking  the 
subspace of holomorphic sections.
This polarization is called {\it K\"ahler polarization}.
This means that we only take those sections which 
are constant in anti-holomorphic directions.
The operator of geometric quantization 
with K\"ahler polarization is  defined as
\begin{equation}\label{E:Geq}
Q_f^{(m)}:=\Pim P_f^{(m)}.
\end{equation}
By the surjectivity of the Toeplitz map there exists a function
$f_m$, depending on the level $m$, such that
$Q_f^{(m)}= T_{f_m}^{(m)}$.
The Tuynman lemma \cite{Tuyn} gives
\begin{equation}\label{E:tuyn}
Q_f^{(m)}=\iu\cdot T_{f-\frac 1{2m}\Delta f}^{(m)},
\end{equation}
where $\Delta$ is the Laplacian with respect to the K\"ahler metric given by
$\omega$.
It should be noted that for \refE{tuyn} the compactness of $M$ is essential.

As a consequence, which will be used later, 
the operators $Q_f^{(m)}$ and the $\Tfm$ have the same asymptotic 
behavior for $m\to\infty$. 
\end{remark}
\begin{remark}
(The non-compact situation.)
If our K\"ahler manifold is not necessarily compact 
then in a first step we consider as quantum space 
the space of bounded holomorphic sections $\ghmb$.
Next we have to restrict the space of quantizable functions
to a subspace of $\Cim$ such that the quantization map \refE{BT}
(now restricted) will
be well-defined.
One possible choice is the subalgebra of functions with
compact support.
After these restrictions the Berezin-Toeplitz operators are defined
as above. In the case of $M$ compact, everything reduces to the
already given objects.
Unfortunately, there is no general result like \refT{bms} valid for
arbitrary quantizable K\"ahler manifolds 
(e.g. for non-compact ones).
There are corresponding results for special important examples. 
But they are more or less shown by case by case studies of the
type of examples using tools exactly adapted to this situation.
See \cite{SchlBer} for references in this respect.
\end{remark}
\begin{remark}
(Auxiliary vector bundle.)
We return to the compact manifold case. It is also
possible to generalize the situation by considering an additional
auxiliary hermitian holomorphic line bundle $E$. The sequence of
quantum spaces is now the space of holomorphic sections
of the bundles $E\otimes L^m$. For the case that $E$ is a line bundle
this was done, e.g. by Hawkins \cite{Haw}, for the general
case by Ma and Marinescu, see  \cite{MaMarRev} for the details.
See also Charles \cite{Charhf}.
By the hermitian structure of $E$ we have a scalar product and
a corresponding projection operator from the space of all
sections to the space of holomorphic sections. 
The Toeplitz operator $\Tfm$ is defined for $f\in C^\infty(M,End(E))$.
The situation considered in this review is that $E$ equals the trivial
line bundle.
But similar results can be obtained in the more general set-up.
This is also true with respect to the star product discussed in
\refS{star}%
\footnote{ For $E$ not a line bundle the Berezin-Toeplitz star product
is a star product in $C^\infty(X,End(E))[[\nu]]$. This might be
considered
as a quantization with additional internal degrees of freedom, see
\cite[Remark 2.27]{MaMarRev}.}.
Of special importance, beside the trivial bundle case, is the case
when
the auxiliary vector bundle is a square root $L_0$ of the canonical
line bundle $K_M$, i.e. $L_0^{\otimes 2}=K_M$ (if the square root exists).
Recall that  $K_M=\bigwedge^n\Omega_M$, where $n=\dim_{\C}M$ and
$\Omega_M$ is the rang $n$ vector bundle of holomorphic
1-differentials.
The corresponding quantization is called quantization with metaplectic
corrections. It turns out that with the metaplectic correction the
quantization behaves better under natural constructions.
An example is the {\it Quantization Commutes with Reduction}
problem in the case that we have a well-defined action of a 
group $G$ on the compact (quantizable) K\"ahler manifold 
with $G$-equivariant quantum line bundle.
Under suitable conditions on the action we have a
linear isomorphy of the $G$-invariant subspace of the quantum spaces 
$\Ho^0(M,L^m)^G$ with the quantum spaces $\Ho^0(M//G,(L//G)^m)$. 
This was shown by Guillemin and Sternberg \cite{GS}.
But this
isomorphy is not unitary. If one uses the quantum spaces with
respect to the metaplectic correction then at least it is
asymptotically (i.e. $m\to\infty$) unitary.
This was shown independently%
\footnote{I am grateful to Xiaonan Ma for pointing this out to me.}
  and with slightly different aspects by
Ma and Zhang \cite{MaZh} (partly based on work of Zhang \cite{Zh})
and by  Hall and Kirwin \cite{HaKir}. See also  
\cite{MaMarBook}. 
For  interesting details about  these approaches see also the
article of Kirwin \cite{Kir} explaining some of the  relations.
For the general singular situation, see  Li
\cite{Li}.

Another case when the quantization with metaplectic correction is more
functorial is if one considers families of K\"ahler manifolds as they
show up e.g. in the context of deforming complex structures on
a given symplectic manifold. See work by 
Charles \cite{Charl} and Andersen, Gammelgaard and Lauridsen \cite{AnGam}.
\end{remark}
\section{Deformation quantization -- star products}\label{S:star}
\subsection{General definitions}\label{SS:star}
There is another approach to quantization. One deforms the commutative
algebra of functions ``into non-commutative directions given by
the Poisson bracket''. It turns out that this can only be done
on the formal level. One obtains a deformation
quantization, also called star product.
This notion was around quite a long time. 
See e.g. Berezin \cite{Berequ},\cite{Bergc},
Moyal \cite{Moy}, Weyl \cite{Weyl}, etc.
Finally, the notion was formalized in 
\cite{BFFLS}. 
See \cite{DiSt} for some historical remarks.

For a 
given  {Poisson algebra} $(\Cim, \cdot, \{\,,\,\})$ of
smooth functions on a manifold $M$,
a {\it star product} for $M$ is an {associative product}
$ \star$ on  $\A:=\Cim[[\nu]]$,
the space of formal power series with coefficients from
$\Cim$,
such that 
for $f,g\in\Cim$
\begin{enumerate}
\item
${f\star g= f \cdot g} {\mod \nu}$,
\item
${\left( f\star g-g\star f\right)/\nu
= -\mathrm{i}\{f,g\}}  {\mod \nu}$.
\end{enumerate}
The star product of two functions $f$ and $g$ 
can be expressed  as 
\begin{equation}\label{E:stara}
f\star g=\sum_{k=0}^\infty \nu^kC_k(f,g),\qquad C_k(f,g)\in\Cim,
\end{equation}
and is extended $\C[[\nu]]$-bilinearly.
It is called 
{differential} (or local) if
the $C_k(\,,\,)$ are bidifferential 
operators with respect to their entries.
If nothing else is said one requires $1\star f=f\star 1=f$, which 
is also called ``null on constants''.

\medskip
\begin{remark}(Existence)
Given a Poisson bracket, is there always a star product?
In the usual setting of deformation theory there always exists 
a trivial deformation. This is not the case here, as the trivial
deformation of $\Cim$ to  $\mathcal{A}$
extending
the point-wise product trivially to the power series, is not
allowed as it does not fulfill the second condition for the
commutator of being a star product
(at least not if the Poisson bracket is non-trivial).
In fact the existence problem is highly non-trivial.
In the symplectic case different existence proofs, from different perspectives,
were given by DeWilde-Lecomte \cite{DeWiLe},
Omori-Maeda-Yoshioka \cite{OMY},
and 
Fedosov \cite{Fed}.
The general Poisson case was settled by Kontsevich
\cite{Kont}.
For more historical information see the review 
\cite{DiSt}.
\end{remark}

\medskip

Two star products
$\star$ and $\star'$ for the same Poisson structure
are  called {\it equivalent} 
if and only if  there exists   a formal series of 
linear operators
$$
B=\sum_{i=0}^\infty B_i\nu^i,\qquad
B_i:\Cim\to\Cim,
$$
with $\ B_0=id\ $ such that
\quad {$ B(f)\star' B(g)=B(f\star g)$}.

To every equivalence class of a differential star product 
its {\it Deligne-Fedosov class} can be assigned. It is a formal de Rham class
of the form 
\begin{equation}\label{E:DFcl}
cl(\star)\in\frac {1}{\mathrm{i}}(\frac {1}{\nu}[\omega]
+\mathrm{H}^2_{dR}(M,\C)[[\nu]]).
\end{equation}
This assignment gives a 1:1 correspondence between
equivalence classes of star products and such formal forms.

\medskip

In the K\"ahler case we might look for star products adapted to
the complex structure.
Karabegov \cite{Kar1} introduced the notion of star products with
{\it separation of variables type} for differential star products.
The star product is of this type if in $C_k(.,.)$ for $k\ge 1$ 
the first argument is only differentiated in holomorphic 
and the second argument in anti-holomorphic directions.
Bordemann and Waldmann in their construction
\cite{BW} used the name {\it star product of Wick type}.%
\footnote{In Karabegov's original approach the role of holomorphic
and antiholomorphic variables are switched, i.e. in
the approach of Bordemann-Waldmann they are of anti-Wick type.
Unfortunately we cannot simply retreat to one these conventions,
as we really have to deal in the following with naturally defined
star products and relations between them, which are of separation
of variables type of both conventions.}
All such star products $\star$ are uniquely given 
(not only up to equivalence) by
their Karabegov form $kf(\star)$ which is a 
formal closed $(1,1)$ form.
We will return to it in \refSS{kara}

\subsection{The Berezin-Toeplitz deformation quantization}\label{SS:btq}
\begin{theorem}\label{T:star}
\cite{BMS},\cite{Schlbia95},\cite{Schlhab},\cite{Schldef},\cite{KS}
There exists  a unique differential  star 
\newline product 
\begin{equation}\label{E:star}
f\star_{BT} g=\sum \nu^kC_k(f,g)
\end{equation}
such that
\begin{equation}\label{E:starapp}
T_f^{(m)}T_g^{(m)}\sim \sum_{k=0}^\infty 
\left(\frac 1m\right)^k  T_{C_k(f,g)}^{(m)}.
\end{equation}
This star product is of 
separation  of variables type
with classifying {Deligne-Fedosov class} $cl$
and {Karabegov form} $kf$ 
\begin{equation}\label{E:BTclasses}
cl(\star_{BT})=
\frac {1}{\mathrm{i}}(\frac {1}{\nu}[\omega]-\frac {\delta}{2}),
\qquad
kf(\star_{BT})=\frac {-1}{\nu}\omega+\omega_{can}.
\end{equation}
\end{theorem}
\noindent
First, the asymptotic expansion in \refE{starapp} has to be understood
in a strong operator norm sense.
For  $f,g\in\Cim$ and for every $N\in\N$  we have
with suitable constants $K_N(f,g)$ for all $m$
\begin{equation}
\label{E:sass}
||T_{f}^{(m)}T_{g}^{(m)}-\sum_{0\le j<N}\left(\frac 1m\right)^j
T_{C_j(f,g)}^{(m)}||\le K_N(f,g) \left(\frac 1m\right)^N\ .
\end{equation}

Second,
the used forms, resp. classes are defined as follows.
 Let $K_M$ be the 
canonical line bundle of $M$, i.e. the $n^{th}$ exterior power of
the holomorphic bundle of 1-differentials. The canonical class $\delta$ is
the first Chern class of this line bundle, i.e. 
$\delta:= c_1(K_M)$.
If we take in $K_M$ the fiber metric coming from the 
Liouville form $\Omega$ then this defines a unique
connection and further a unique curvature $(1,1)$-form
$\w_{can}$. 
In our
sign conventions we have  $\delta=[\w_{can}]$.
The Karabegov form will be introduced in \refSS{kara}.

\begin{remark}
Using \refT{bms} and the Tuynman relation \refE{tuyn} one can show that
there exists a {star product $\star_{GQ}$}
given by asymptotic expansion of the product of geometric quantization
operators.
The star product 
$\star_{GQ}$ is {equivalent} to $\star_{BT}$, via the equivalence
$B(f):=(id-\nu\frac {\Delta}{2})f$.
In particular, it has the same Deligne-Fedosov class. But it 
is {not} of separation of variables type, see \cite{Schldef}. 
\end{remark}
\subsection{Star product of separation of variables type}\label{SS:kara}
%
In \cite{Kar1,Kar2} Karabegov not only gave the notion of
{\it separation of variables type}, but also a proof of 
existence of such formal star products for any 
K\"ahler manifold, whether compact, non-compact, quantizable, or
non-quantizable.
Moreover, he classified them completely as individual 
star product not only up to equivalence.

In this set-up it is quite useful to consider more generally 
pseudo-K\"ahler manifolds
$(M,\omega_{-1})$, i.e.
complex manifolds with a non-degenerate closed $(1,1)$-form $\omega_{-1}$ not
necessarily positive.
(In this context 
it is convenient to denote by $\w_{-1}$ the $\w$ we use
at other places of the article.)

A formal form 
\begin{equation}\label{E:defw}
\widehat{\omega}=
(1/\nu)\omega_{-1}+\omega_0+\nu\omega_1+\dots
\end{equation}
is
called a formal deformation of the form $(1/\nu)\omega_{-1}$ if the forms
$\omega_r,\ r\geq 0$, are closed but not necessarily nondegenerate
(1,1)-forms on $M$.
Karabegov showed that to every such $\widehat{\omega}$ there 
exists a star product $\star$. Moreover he showed
that all deformation quantizations with separation
of variables on the  pseudo-K\"ahler manifold $(M,\omega_{-1})$ are
bijectively parameterized by the formal deformations of the form
$(1/\nu)\omega_{-1}$. 
By definition the  {\it Karabegov form} 
$kf(\star):=\widehat{\omega}$, i.e.  it is 
taken to be the  $\widehat{\omega}$ defining $\star$.

Let us indicate the principal idea of the construction.
First, assume that we have such a star product 
$(\mathcal{A}:=\cim[[\nu]],\star)$. Then 
for $f,g\in\mathcal{A}$ the  operators of left
and right multiplication  $L_f,R_g$  are given by  
$L_fg=f\star g=R_gf$. The
associativity of the star-product $\star$ is equivalent to the
fact that $L_f$ commutes with $R_g$ for all $f,g\in{\mathcal{A}}$.
If a star  product is differential then 
$L_f,R_g$ are formal differential operators.
Now 
Karabegov constructs his star product 
associated to the deformation $\widehat{\w}$ in the following way.
First he chooses on every 
contractible
coordinate chart $U\subset M$ (with holomorphic
coordinates $\{z_k\}$)
its formal potential  
\begin{equation}\label{E:formp}
\widehat{\Phi}=(1/\nu)\Phi_{-1}+\Phi_0+\nu\Phi_1+\dots,
\qquad
\widehat{\omega}=i\partial\bar\partial\widehat{\Phi}.
\end{equation}
Then the construction is done in such a way that 
the left (right) multiplication operators 
$L_{\partial\widehat{\Phi}/\partial z_k}$ ($R_{\partial\widehat{\Phi}/\partial\bar z_l}$)
on $U$
are realized as formal differential operators
\begin{equation}
L_{\partial\widehat{\Phi}/\partial z_k}
=\partial\widehat{\Phi}/\partial z_k+\partial/\partial z_k, 
\quad\text{and}\quad 
R_{\partial\widehat{\Phi}/\partial\bar z_l}
=\partial\widehat{\Phi}/\partial\bar z_l+\partial/\partial\bar z_l.
\end{equation} 
The set
$\mathcal{L}(U)$ of all left multiplication
operators on $U$ is completely described as the
set of all formal differential operators
commuting with the point-wise multiplication
operators by antiholomorphic coordinates
$R_{\bar z_l}=\bar z_l$ and the operators
$R_{\partial\widehat{\Phi}/\partial\bar
z_l}$. From the knowledge of $\mathcal{L}(U)$ the 
star product on $U$ can be reconstructed.
This follows from the simple fact that $L_g(1)=g$ and 
$L_f(L_g)(1)=f\star g$.
The operator corresponding to the left multiplication
with the (formal) function $g$ can recursively (in the
$\nu$-degree)
be calculated from the fact that it commutes with 
the operators
 $R_{\partial\widehat{\Phi}/\partial\bar z_l}$.
The local
star-products agree on the intersections of the
charts and define the global star-product
$\star$ on $M$.
See the original work of Karabegov \cite{Kar1} for these
statements.

We have to mention that  this  original
 construction of Karabegov will yield a star product of separation of
variables type but with the role of holomorphic and antiholomorphic
variables switched.
This says 
  for any open
subset $U\subset M$ and any holomorphic function $a$ and
antiholomorphic function $b$ on $U$ the operators $L_a$ and
$R_b$ are the operators of point-wise multiplication by $a$ and
$b$ respectively, i.e., $L_a=a$ and $R_b=b$. 

The construction of Karabegov is on one side very universal without
any restriction on the (pseudo) K\"ahler manifold. But it does
not establish any connection to an operator representation.
The existence of such an operator representation is related in a
vague sense to the quantization condition.
The BT deformation quantization has such a relation and
singles out a unique star product. Modulo switching the role
of holomorphic and anti-holomorphic variable $\star_{BT}$ corresponds
to a unique Karabegov form. This form is given in \refE{BTclasses}.
The identification is done in \refSS{ident} further down.
That the form starts with $(-1/\nu)\omega$ is due to the fact that
the role of the variables have to be switched to end up
in Karabegov's classification.

\subsection{Karabegov's formal Berezin transform}\label{SS:Kbere}
%
Given  a pseudo-K\"ahler manifold $(M,\omega_{-1})$.
In the frame of his construction and classification  
Karabegov 
assigned to each star products $\star$ with the separation of
variables property the formal
{\it Berezin transform} $I_\star$. 
It is 
as the unique formal differential operator on
$M$ such that for any open subset $U\subset M$,
antiholomorphic functions $a$ and holomorphic
functions $b$ on $U$ the relation 
\begin{equation}\label{E:formd}
a\star b=I(b\cdot a)=I(b\star a),
\end{equation}
holds true.
The last equality is automatic and  is due to the fact, 
that by the separation of variables property $b\star a$ is the
point-wise
product $b\cdot a$. 
He shows 
\begin{equation}
I=\sum_{i=0}^{\infty} I_i\;\nu^i, \quad  I_i:\cim\to\cim,
\quad I_0=id,\quad I_1=\Delta .
\end{equation}
Let us summarize. Karabegov's classification gives 
for a fixed pseudo-K\"ahler manifold a
1:1 correspondence 
between 
\newline
(1) the set of star products with separation of variables
type in Karabegov convention and
\newline
(2) the set of formal deformations \refE{defw} of $\omega_{-1}$.
\newline
Moreover, the formal Berezin transform $I_\star$
determines the $\star$ uniquely.

\medskip
We will introduce further down a Berezin transform in 
the set-up of the BT quantization.
In \cite{KS} it is shown that its asymptotic expansion gives
a formal Berezin transform in the sense of Karabegov, associated to
a star product related to $\star_{BT}$ explained as 
follows.

\subsection{Dual and opposite star products}\label{SS:kdual}
Given for the pseudo-K\"ahler manifold $(M,\omega_{-1})$ 
a star product $\star$ of separation of variables type
(in Karabegov convention)
Karabegov defined with the help of $I=I_\star$  
the following associated star products. 
First the {\it dual}
star-product $\tilde\star$ on $M$ is defined for
$f,g\in \mathcal{A}$ by the formula 
\begin{equation}
f\,\tilde\star\, g=I^{-1}(I(g)\star I(f)).
\end{equation}
It is  a star-product with separation of
variables but now on the pseudo-K\"ahler manifold
$(M,-\omega_{-1})$.
Denote by $\tilde\omega=-(1/\nu)\omega_{-1}
+\tilde\omega_0+\nu\tilde\omega_1+\dots$ the
formal form parameterizing the star-product
$\tilde\star$.  By definition $\tilde\omega=kf(\tilde\star)$.  
Its  formal Berezin transform equals
$I^{-1}$, and thus the dual to $\ \tilde\star\ $ is
again $\ \star\ $.

Given a star product, the opposite star product is obtained
by switching the arguments. Of course the sign of the
Poisson bracket is changed.
Now we take the 
opposite of the dual
star-product, $\star'=\tilde\star^{op}$, given
by 
\begin{equation}\label{E:bteq}
f\star' g=g\,\tilde\star\, f=I^{-1}(I(f)\star I(g)).
\end{equation}
It  defines a deformation quantization with
separation of variables on $M$, but with the
roles of holomorphic and antiholomorphic
variables swapped - in contrast  to $\star$. 
But now the pseudo-K\"ahler 
manifold will be $(M,\omega_{-1})$.
Indeed  
the formal Berezin
transform $I$ establishes an equivalence of
the deformation quantizations $(\mathcal{A},\star)$ and
$(\mathcal{A},\star')$. 

\medskip
How is the relation to the Berezin-Toeplitz star product 
$\star_{BT}$ of \refT{star}?
There exists a certain formal deformation 
$\widehat{\w}$ of the form
$(1/\nu) \w$ which yields a star product $\star$ in the Karabegov
sense
\cite{KS}.
The opposite of its dual will be equal to the 
Berezin-Toeplitz star product, i.e. 
\begin{equation}
\star_{BT}\ =\ \tilde\star^{op}\ =\ \star '\ .
\end{equation}
The classifying Karabegov form $kf(\tilde\star)$  
will be the form \refE{BTclasses}.
Here we fix the convention that we take for determining
the Karabegov form of the BT star product the Karabegov form of
the opposite one to adjust to Karabegov's original convention, i.e.
\begin{equation}
kf(\star_{BT}):=kf(\star_{BT}^{op})=kf(\tilde\star).
\end{equation}
As $\tilde\star$ is a star product for the pseudo-K\"ahler manifold
$(M,-\omega)$ the $kf(\star_{BT})$   starts with $(-1/\nu)\omega$.

The formula \refE{bteq} gives an equivalence between
$\ \star\ $ and $\ \star_{BT}\ $  via $\ I\ $. 
Hence, we have for the Deligne-Fedosov class 
 $cl(\star)=cl(\star_{BT})$, see the formula \refE{BTclasses}.
We will identify $\widehat{\w}=kf(\star)$ in \refSS{ident}.
\section{Global Toeplitz operators}\label{S:globToe}
In this section we will indicate some parts of the proofs of
\refT{star} and \refT{bms}.
For this goal we consider 
the bundles $L^m$ over the compact K\"ahler
manifold $M$ as associated line bundles of one unique 
$S^1$-bundle over $M$.
The Toeplitz operator will appear as ``modes'' of a global Toeplitz
operator.
Moreover, we will need the same set-up to discuss coherent states, 
Berezin symbols, and the Berezin transform in the next sections.

\subsection{The disc bundle}\label{SS:disc}
%
Recall that our quantum line bundle $L$ was assumed to be  already
very ample.
We pass   to its {dual} line bundle
$\ (U,k):=(L^*,h^{-1})\ $ 
with dual metric $k$.
In the example of the projective space, 
the
quantum line bundle is the hyperplane section bundle and its
dual is the tautological line bundle.
Inside  the total space $U$, we consider the {circle bundle}
$$
Q:=\{\lambda\in U\mid k(\lambda,\lambda)=1\},
$$
and denote by 
$\tau:Q\to M$ (or $\tau:U\to M$)  the {projections}
to the base manifold $M$.

The bundle $Q$ is a {contact manifold}, i.e. there is a 1-form $\nu$
such that
\newline
$\mu=\frac 1{2\pi}\tau^*\Omega
\wedge \nu$ is a volume form
on $Q$. Moreover,
\begin{equation}\label{E:discint}
\int_Q(\tau^*f)\mu=\int_Mf\,\Omega,\qquad \forall f\in\Cim.
\end{equation}
Denote by $\Lqv$ the corresponding $L^2$-space 
on $Q$. 
Let {$\mathcal{H}$}  be the space of (differentiable)
functions on $Q$ which can be 
extended to holomorphic functions on the disc bundle 
(i.e. to the ``interior'' of the circle bundle),
and 
{$\Hm$} the subspace of $\mathcal{H}$ consisting of
  $m$-homogeneous functions on $Q$.
Here  
{$m$-homogeneous} means 
$\psi(c\la)=c^m\psi(\la)$.
For further reference let us introduce the
following (orthogonal) projectors:
 the 
 {\it  Szeg\"o projector} 
\begin{equation}\label{E:szeg}
\Pi:\Lqv\to \Hc, 
\end{equation}
and its components the {\it Bergman projectors}
\begin{equation}
\pimh:\Lqv\to\Hm.
\end{equation}

The bundle $Q$ is a {$S^1-$bundle}, and the $L^m$ are {associated} line bundles.
The sections of  $L^m=U^{-m}$ are 
identified with those functions $\psi$ on $Q$ 
which are {homogeneous of degree $m$}.
This identification is given
on the level of the $\mathrm{L}^2$ spaces 
by the map
\begin{equation}
\gamma_m:\Lpm \to \Lqv,\quad {s\mapsto \psi_s}\quad\text{where}
\end{equation}
\begin{equation}\label{E:ident}
\psi_s(\alpha)=\alpha^{\otimes m}(s(\tau(\alpha))).
\end{equation}
Restricted to the holomorphic sections we obtain the 
{unitary} isomorphism  
\begin{equation}\label{E:iso}
\gamma_m:\ghm \cong \Hm.
\end{equation}
\subsection{Toeplitz structure}\label{SS:toeplitz}
%
Boutet de Monvel  and
Guillemin introduced 
the notion of a Toeplitz structure $(\Pi,\Sigma)$ and associated 
generalized Toeplitz operators \cite{BGTo}.
If we specialize this to our situation
then 
$\Pi $ is the  Szeg\"o projector  \refE{szeg}
and $\Sigma$  is  the submanifold 
\begin{equation}
\Sigma:=\{\;t\nu(\lambda)\;|\;\lambda\in Q,\,t>0\ \}\ \subset\  T^*Q
\setminus 0
\end{equation}
of the tangent bundle of $Q$ 
defined with the help of the 1-form $\nu$.
It turns out that  $\Sigma$ is a 
symplectic submanifold,  a symplectic cone.

A (generalized) {\em Toeplitz operator} of order $k$ is  an operator
$A:\Hc\to\Hc$ of the form
$\ A=\Pi\cdot R\cdot \Pi\ $ where $R$ is a
pseudo-differential operator ($\Psi$DO)
of order $k$ on
$Q$.
The Toeplitz operators constitute a ring.
The  symbol of $A$ is the restriction of the
principal symbol of $R$ (which lives on $T^*Q$) to $\Sigma$.
Note that $R$ is not fixed by $A$, but Boutet de Monvel and
Guillemin showed that the  symbols
are well-defined and that they obey the same rules as the
symbols of   $\Psi$DOs.
In particular, the following relations are valid:
\begin{equation}
\label{E:symbol}
\sigma(A_1A_2)=\sigma(A_1)\sigma(A_2),\qquad
\sigma([A_1,A_2])=\iu\{\sigma(A_1),\sigma(A_2)\}_\Sigma.
\end{equation}
Here $\{.,.\}_{\Sigma}$ is the restriction of the canonical
Poisson structure of $T^*Q$ to $\Sigma$ coming from the
canonical symplectic form  on $T^*Q$.
Furthermore, a Toeplitz operator of order $k$ with vanishing
symbol is a Toeplitz operator of order $k-1$.

We will need the following 
two generalized Toeplitz operators:

(1) The generator of the circle action
gives the  operator $D_\varphi=\dfrac 1{\iu}\dfrac {\partial}
{\partial\varphi}$, where $\varphi$ is the angular variable. 
It is an operator of order 1 with symbol $t$.
It operates on $\Hm$ as multiplication by $m$.

(2) For $f\in\Cim$ let $M_f$ be the  operator on
$\Lqv$ 
corresponding to multiplication with $\tau^*f$.
We set
\begin{equation}
T_f=\Pi\cdot M_f\cdot\Pi:\quad \Hc\to\Hc\ .
\end{equation}
As $M_f$ is constant along the fibers of $\tau$, 
the operator $T_f$ 
commutes with the circle action.
Hence we can decompose
\begin{equation}
T_f=\prod\limits_{m=0}^\infty\Tfm\ ,
\end{equation}
where $\Tfm$ denotes the restriction of $T_f$ to $\Hm$.
After the identification of $\Hm$ with $\ghm$ we see that these $\Tfm$
are exactly the Toeplitz operators  $\Tfm$ introduced in \refS{btq}.
We call   $T_f$   the global Toeplitz operator and
the $\Tfm$ the local Toeplitz operators.
The operator $T_f$ is  of order $0$.
Let us denote by
$\ \tau_\Sigma:\Sigma\subseteq T^*Q\to Q\to M$ the composition
then we obtain for its symbol  $\sigma(T_f)=\tau^*_\Sigma(f)$.
\subsection{The construction of the BT star product}\label{SS:const}
%
To give a sketch of the 
proof of \refT{star} we will need the statements of \refT{bms}.
The part (a) of this theorem we will show 
with the help of the asymptotic expansion of the
Berezin transform in \refSS{norm}.
The other parts will be sketched here, too.
Full proofs of \refT{star} can be found in \cite{Schldef},
\cite{Schlhab}.
Full proofs of \refT{bms} in \cite{BMS}.

Let the notation be as  in the last subsection.
In particular, let $T_f$ be the 
Toeplitz operator, $D_\varphi$ the operator of rotation, and 
$\Tfm$, resp. $(m\cdot)$ their projections on the eigenspaces 
$\Hm\cong\ghm$.

\medskip
{\bf (a) The definition of the  $C_j(f,g)\in \Cim$}
\newline
The construction is done inductively in such a way
that 
\begin{equation}
\label{E:Asum}
A_N=D_\varphi^N T_fT_g- \sum_{j=0}^{N-1}
D_\varphi^{N-j}T_{C_j(f,g)}
\end{equation}
is always a  Toeplitz operator of order zero.
The operator $A_N$ is $S^1$-invariant, i.e.
$D_\varphi\cdot A_N=A_N\cdot D_\varphi$. 
As it is of order zero his symbol 
is a function on $Q$. By the $S^1$-invariance the symbol is even 
given by (the pull-back of) a function on $M$.
We take  this function as  next element $C_N(f,g)$
in the star product.
By construction, the 
operator  $\ A_N-T_{C_N(f,g)}\ $ is of order   $-1$ and 
$A_{N+1}=D_\varphi(A_N-T_{C_N(f,g)})$ is of order  0 
and exactly of the form given 
in \refE{Asum}.

The  induction starts with
\begin{gather}
A_0=T_fT_g, \qquad\mathrm{and}\
\\
\sigma(A_0)=\sigma(T_f)\sigma(T_g)=
\tau^*_{\Sigma}(f)\cdot 
\tau^*_{\Sigma}(g)
=
\tau^*_{\Sigma}(f\cdot g)\ .
\end{gather}
Hence, $C_0(f,g)=f\cdot g$ as required.
\newline
It remains to show statement \refE{sass}
about the asymptotics.
As an operator of order zero on a compact manifold
$A_N$ is bounded 
($\Psi$DOs of order 0 on compact manifolds are bounded).
By the  $S^1$-invariance  we can write
$A=\prod_{m=0}^\infty A^{(m)}$
where $A^{(m)}$ is the restriction of $A$ on the 
orthogonal subspace $\Hm$.
For the norms we get $\ ||A^{(m)}||\le ||A||$.
If we calculate the restrictions we obtain
\begin{equation}
\label{E:azwi}
||m^N\Tfm\Tgm-\sum_{j=0}^{N-1}m^{N-j}T^{(m)}_{C_j(f,g)}||
=||A_N^{(m)}||\le ||A_N||\ .
\end{equation}
After dividing by $m^N$ Equation \refE{sass} follows.
Bilinearity is clear.
For $N=1$ we obtain \refE{prod} and \refT{bms}, Part (c).

\medskip
{\bf (b) The Poisson structure}

First we  sketch the proof for \refE{dirac}.
For a fixed $t>0$ 
\begin{equation}
\Sigma_t:=\{t\cdot \nu(\la)\mid \la\in Q\}\quad \subseteq \Sigma.
\end{equation}
It turns out that 
 $\ {\omega_\Sigma}_{|\Sigma_t}=-t\tau_\Sigma^*\omega\ $.
The commutator
$[T_f,T_g]$ is a  Toeplitz operator of order $-1$.
From the above  we obtain
with \refE{symbol} for the 
symbol of the commutator 
\begin{equation}
\sigma([T_f,T_g])(t\nu(\lambda))=\iu\{\tau_\Sigma^* f,\tau_\Sigma^*g
\}_\Sigma(t\nu(\lambda))=
-\iu t^{-1}\{f,g\}_M(\tau(\lambda))\ .
\end{equation}
We consider the Toeplitz operator
\begin{equation}
A:=D_\varphi^2\,[T_f,T_g]+\iu D_\varphi\, T_{\{f,g\}}\ .
\end{equation}
Formally this is an operator of order 1.
Using $\ \sigma(T_{\{f,g\}})=\tau^*_\Sigma \{f,g\}$ 
and $\sigma(D_\varphi)=t$ we see that its principal
symbol vanishes. Hence
it is an operator of order 0.
Arguing as above we consider its components $A^{(m)}$
and get $\ ||A^{(m)}||\le ||A||$.
Moreover, 
\begin{equation}
A^{(m)}=A_{|\Hm}=m^2[\Tfm,\Tgm]+\iu m\Tfgm.
\end{equation}
Taking the norm bound and dividing it by $m$ we get 
part (b) of \refT{bms}.
Using \refE{iso} the norms involved indeed coincide.

\medskip
For the star product we have to show that
$C_1(f,g)-C_1(g,f)=-\iu\{f,g\}$.
We write explicitly  \refE{azwi} 
for  $N=2$ and the pair of functions
$(f,g)$:
\begin{equation}
||m^2 \Tfm\Tgm-m^2 T^{(m)}_{f\cdot g}-m T^{(m)}_{C_1(f,g)}||
\le K\ .
\end{equation}
A corresponding expression is obtained for the pair  $(g,f)$.
If we subtract both operators inside of the norm 
we obtain (with a suitable $K'$)
\begin{equation}
||m^2 (\Tfm\Tgm-\Tgm\Tfm)-m (T^{(m)}_{C_1(f,g)}-T^{(m)}_{C_1(g,f)})||
\le K'\ .
\end{equation}
Dividing by $m$ and multiplying with $\iu$ we obtain
\begin{equation}
||m\iu [\Tfm,\Tgm]-T^{(m)}_{\iu\big(C_1(f,g)-C_1(g,f)\big)}||
=O(\frac 1m)\ .
\end{equation}
Using the asymptotics given by \refT{bms}(b) for 
the commutator we get
\begin{equation}
|| T^{(m)}_{\{f,g\}-\iu\big(C_1(f,g)-C_1(g,f)\big)}||=O(\frac 1m)\ .
\end{equation}
Taking the limit for $m\to\infty$ and using 
\refT{bms}(a) we get
\begin{equation}\ ||\{f,g\}-\iu(C_1\big(f,g)-C_1(g,f)\big)||_\infty=0\ . 
\end{equation}
Hence indeed, $\ \{f,g\}=\iu(C_1(f,g)-C_1(g,f))$.
For the associativity and further results, see \cite{Schldef}.

\bigskip
Within this approach the calculation of the coefficient
functions $C_k(f,g)$ is recursively and not really 
constructive. In \refSS{calc} we will show another way 
how to calculate the coefficients. It is based on the
asymptotic expansion of the Berezin transform, which itself
is obtained via the off-diagonal expansion of the
Bergman kernel.

In fact the Toeplitz operators again can be expressed via
kernel functions also  related to the Bergman kernel.
In this way certain extensions of the presented results
are possible. See in particular work by Ma and 
Marinescu for compact symplectic manifolds and orbifolds.
One might consult the review 
\cite{MaMarRev} for results and further references.

For another approach (still symbol oriented) to
Berezin Toeplitz operator and star product quantization see
Charles \cite{Char}, \cite{Chardiss}.

 \section{Coherent  states and symbols}\label{S:coh}
Berezin constructed for an important but limited classes of
K\"ahler manifolds a star product. The construction was based
on his covariant symbols given for domains in $\C^n$.
In the following we will present their definition for
arbitrary compact quantizable K\"ahler manifolds.
\subsection{Coherent states}\label{SS:cohs}
%
We look again at the relation \refE{ident}
\begin{equation*}
\psi_s(\alpha)=\alpha^{\otimes m}(s(\tau(\alpha))),
\end{equation*}
but now from the 
point of view of the linear evaluation functional.
This means,
we {fix}
${\al}\in U\setminus 0$ and vary the sections $s$. 

The {\it  coherent vector (of level m)}  associated to
the point $\al\in U\setminus 0$ is the  element {$\eam$}
of $\ghm$ with 
\begin{equation}\label{E:cohv}
\skp{{\eam}}{s}=\psi_s(\alpha)=\alpha^{\otimes m}(s(\tau({\alpha})))
\end{equation}
for all $s\in\ghm$.
A direct verification shows
$e_{c\alpha}^{(m)}={\bar c}^m\cdot \eam$ for
$c\in\C^*:=\C\setminus\{0\}$. 
Moreover, as the bundle is very ample we get $\eam\ne 0$.

This allows 
the following 
definition.

\begin{definition}\label{D:cohs}
The {\it  coherent state (of level m)}  associated to
$x\in M$ is  the projective class
\begin{equation}\label{E:cohs}
{\e^{(m)}_x}:= [\eam]\in\Pro(\ghm),\qquad \al\in\tau^{-1}(x), \al\ne 0.
\end{equation}
\end{definition}
The {\it  coherent state embedding} is the 
antiholomorphic embedding
\begin{equation}\label{E:cohse}
M\quad \to\quad \Pro(\ghm)\ \cong\ \pnc[N],
\qquad 
x\mapsto [e^{(m)}_{\tau^{-1}(x)}].
\end{equation}
See \cite{BerSchlcse} for some geometric properties of the
coherent state embedding.

\begin{remark}
A coordinate independent version of Berezin's original
definition and extensions
to line bundles were given by  
Rawnsley \cite{Raw}.
It plays an important role in the work of Cahen, Gutt, and
Rawnsley on the quantization of K\"ahler manifolds 
\cite{CGR1,CGR2,CGR3,CGR4},
 via Berezin's covariant symbols.
In these works 
the coherent vectors are parameterized by the elements of
$L\setminus 0$.
The  definition here uses the points of the total space of the 
dual bundle $U$. It  has the advantage that one can consider
all tensor powers of $L$ together on an equal footing.
\end{remark}
\subsection{Covariant Berezin symbol}\label{S:cov}
\begin{definition}
For an operator $A\in\eghm$ its 
{\it covariant Berezin symbol $\sigma^{(m)}(A)$  (of level $m$)}
is defined as the 
function
\begin{equation}\label{E:covsym}
{\sigma^{(m)}(A)}:M\to\C,\quad
x\mapsto \sm(A)(x):=
\frac {\skp {\eam}{A\eam}}{\skp {\eam}{\eam}},\quad
\alpha\in\tau^{-1}(x)\setminus\{0\}.
\end{equation}
\end{definition}
Using the 
{\it coherent projectors} (with the convenient bra-ket notation) 
\begin{equation}\label{E:cohp}
P^{(m)}_{x}=\frac {|\eam\rangle\langle \eam|}{\langle
  \eam,\eam\rangle},\qquad \alpha\in\tau^{-1}(x)
\end{equation}
it can be rewritten as 
$
\sm(A)=\Tr(AP^{(m)}_x)
$.
In abuse of notation $\alpha\in\tau^{-1}(x)$ should always mean
$\alpha\ne 0$.

\subsection{Contravariant Symbols}\label{SS:contra}
%
We need
{Rawnsley's  epsilon function} $\epsm$ \cite{Raw}
to introduce contravariant symbols in the general
K\"ahler manifold setting. 
It is defined as
\begin{equation}\label{E:rawe}
{\epsm}: M\to\cim,\quad 
x\mapsto {\epsm(x)}:=\frac
{h^{(m)}(\eam,\eam)(x)}
{\skp {\eam}{\eam}},\ \al\in\tau^{-1}(x).
\end{equation}
As $\epsm>0$ we can introduce the {modified 
measure}
$
{\Ome(x):=\epsm(x)\Om(x)
}$
on the  space of functions on $M$.
If $M$ is a homogeneous manifold under a transitive 
group action and everything is invariant,
$\epsm$ will be constant. This was the case 
considered by Berezin.

\begin{definition}
Given an operator  $A\in\eghm$ then {\bf a} 
 { contravariant Berezin symbol} {$\svm(A)\in\cim$
of $A$}  is defined by the {representation} of the operator
$A$ {as an integral}
\begin{equation}\label{E:contra}
A=\int_M\svm(A)(x)P^{(m)}_x\,\Ome(x),
\end{equation}
if such a representation exists.
\end{definition}
We quote from \cite[Prop. 6.8]{SchlBer} that
the {Toeplitz operator $\Tfm$} admits  such a representation  with
${\svm(\Tfm)=f}$. This says,  the function $f$ 
itself is {a} contravariant symbol of the Toeplitz operator
$\Tfm$. Note that the contravariant symbol is not uniquely
fixed by the operator.
As an immediate consequence from the surjectivity of the
Toeplitz map it follows that
every operator $A$
 has a contravariant 
symbol, i.e. every operator $A$ has a representation \refE{contra}.
For this we have to keep in mind, that our K\"ahler manifolds are
compact.

\medskip

Now we introduce
 on $\eghm$  the {Hilbert-Schmidt norm} 
$
\skps {A}{C}{HS}=Tr(A^*\cdot C).
$
In \cite{Schlbia98} (see also \cite{Schlgeoquant}), we showed that 
\begin{equation}\label{E:adj} 
{\skps {A}{\Tfm}{HS}=\skpsm {\sm(A)} {f} {\epsilon} \ .}
\end{equation}
This says that the {Toeplitz map} $f\to \Tfm$ and the {covariant symbol map}
$A\to\sm(A)$ are {adjoint}.
By the adjointness property from the surjectivity of the
Toeplitz map 
the following follows.
\begin{proposition}\label{P:inj}
The covariant symbol map
is injective.
\end{proposition}
Other results following from the adjointness are
\begin{equation}\label{E:trace}
{\tr (\Tfm)}=\int_M f\;\Ome=\int_M\sm(\Tfm)\;\Ome.
\end{equation}
\begin{equation}
{\dim\ghm} =\int_M\Ome=\int_M\epsilon^{(m)}(x)\;\Omega.
\end{equation}
In particular, in the special case that 
{$\epsilon^{(m)}(x)=const$}\quad then 
\begin{equation}
{\epsilon^{(m)}}=\frac {\dim\ghm}{vol_{\Omega}(M)}.
\end{equation}
\subsection{The original Berezin star product}\label{SS:borg}
%
Under very 
restrictive conditions on the manifold
it is possible to construct 
the {\it Berezin star product}
with the help of the covariant symbol map.
This was done by Berezin himself \cite{Berequ},\cite{Beress}
and later by Cahen, Gutt, and Rawnsley 
\cite{CGR1}\cite{CGR2}\cite{CGR3}\cite{CGR4} for more examples.
We will indicate this in the following.

Denote by {$\A^{(m)}\ \le \Cim$,} the subspace of
functions which appear as 
level $m$ covariant symbols of operators.
By  \refP{inj}
for the two symbols
 $\sm(A)$ and $\sm(B)$ the 
operators $A$ and $B$ are uniquely fixed.
Hence, it is possible to define 
the deformed product by 
 \begin{equation}
\sm(A)\star_{(m)}\sm(B):=\sm(A\cdot B).
\end{equation}
Now  $\star_{(m)}$ defines on $\A^{(m)}$ 
 an  associative and noncommutative
product.

It is even possible to give an  expression for
the resulting symbol.
For this we introduce 
the {\it two-point function}
\begin{equation}\label{E:two}
\psi^{(m)}(x,y)=
\frac {\langle \eam, \ebm\rangle \langle \ebm,\eam
\rangle }{\langle \eam,\eam\rangle \langle \ebm,\ebm \rangle }
\end{equation}
with $\al=\tau^{-1}(x)$ and $\beta=\tau^{-1}(y)$.
This function is well-defined on $M\times M$.
Furthermore, we have the {\it two-point symbol}
\begin{equation}\label{E:twos}
\sm(A)(x,y)
=\frac {\langle \eam,A \ebm \rangle }
{\langle \eam, \ebm \rangle }.
\end{equation}
It is the analytic extension of the real-analytic covariant symbol.
It is well-defined on an open dense subset of $M\times M$
containing
the diagonal.
Then
\begin{multline}\label{E:bsint}
\sm(A)\star_{(m)}\sm(B)(x)= \sm(A\cdot B)(x)=
\frac {\langle \eam,A\cdot B\, \eam\rangle }
{\langle \eam,\eam\rangle }
\\
=
\frac 1{\langle \eam,\eam\rangle }\int_M
{\langle \eam,A\ebm\rangle \langle \ebm,B\eam\rangle}
\frac {\Ome(y)}{\langle \ebm,\ebm\rangle }
\\
=
\int_M \sm(A)(x,y)\cdot\sm(B)(y,x)\cdot \psi^{(m)}(x,y)\cdot
\Ome(y)\ .
\end{multline}

The crucial problem is how to relate different levels $m$ to define 
for all possible symbols a unique product not depending on $m$.
In certain special situations like those studied by 
Berezin, and Cahen, Gutt and Rawnsley
the subspaces are nested into each other and the union
$\A=\bigcup_{m\in\N}\A^{(m)}$ is a dense subalgebra of $\Cim$.
This is the case if
  the manifold is a homogeneous manifold and 
the epsilon function $\epsm$ is 
a constant.
A detailed analysis shows that in this case  a
star product is given.

For related results see also work of Moreno and Ortega-Navarro 
\cite{MoOr}, \cite{Mor1}.
In particular,  also the work of Engli\v s
\cite{Engbk,Eng2,Eng1,Eng0}. 
Reshetikhin and Takhtajan \cite{ResTak} gave a construction
of a (formal) star product using formal integrals 
(and associated Feynman graphs) in 
the spirit of the Berezin's covariant symbol  construction,
see \refSS{resta}

In \refSS{btstarg}
using the Berezin transform and its properties 
discussed in the next section 
(at least in the case of quantizable
compact K\"ahler manifolds) 
we will 
introduce a star product dual to  the by \refT{star} 
existing $\star_{BT}$. It  
will generalizes the above star product.
\section{The Berezin transform and Bergman kernels}
\label{S:bere}
\subsection{Definition and asymptotic expansion of the Berezin
  transform}\label{SS:bere}
%
\begin{definition}
The map 
\begin{equation}\label{E:btrans}
{I^{(m)}}:\Cim\to\Cim,\qquad
f\mapsto
{\Imb(f)}:=\sm(\Tfm),
\end{equation}
obtained by 
starting with a function $f\in\Cim$, taking its 
Toeplitz operator $\Tfm$, and then calculating the
covariant symbol
is called the {\it Berezin transform (of level $m$)}.
\end{definition}
To distinguish it from the formal Berezin transforms introduced by
Karabegov for any of his star products sometimes we will call the
above the geometric Berezin transform.
Note that it is uniquely fixed by the geometric 
setup of the quantized K\"ahler manifold.

{}From the point of view of Berezin's approach 
the operator $T_f^{(m)}$ has as a  contravariant symbol $f$.
Hence $\Imb$ gives a correspondence between contravariant symbols
and covariant symbols of operators.
The Berezin transform was introduced and studied by 
Berezin \cite{Beress} for certain classical symmetric 
domains in $\C^n$. These results where 
extended by Unterberger and Upmeier \cite{UnUp},
see also Engli\v s \cite{Eng1,Eng2,Engbk} 
and Engli\v s and Peetre \cite{EnPe}.
Obviously, the Berezin transform makes perfect  sense
in the compact K\"ahler case which we
consider here.

\begin{theorem}\label{T:btrans}
\cite{KS}
Given $x\in M$ then the {Berezin transform} 
$\Imb(f)$  has a complete {asymptotic 
expansion} in powers of $1/m$ as $m\to\infty$
\begin{equation}\label{E:btransas}
\Imb(f)(x)\quad\sim \quad\sum_{i=0}^\infty I_i(f)(x)\frac {1}{m^i}\ , 
\end{equation}
where   $I_i:\cim\to\cim$ are linear maps 
given by differential operators, uniformly defined for all $x\in M$.
Furthermore,
$
{I_0(f)=f,\quad I_1(f)=\Delta f}.
$
\end{theorem}
Here $\Delta$ is the  {Laplacian} with respect
to the metric given by the  K\"ahler form $\w$.
By {\it  complete asymptotic expansion}
the following is understood.
Given $f\in\Cim$, $x\in M$ and an $N\in\N$ then there
exists  a positive constant $A$ such that
$$
{\left|\Imb(f)(x)- \sum_{i=0}^{N-1} I_i(f)(x)\frac {1}{m^i}\right|}_{\infty}
\quad\le\quad \frac {A}{m^N}\ .
$$
The proof of  this theorem is quite involved. An important
intermediate step of independent interest is the 
off-diagonal asymptotic expansion of the {Bergman kernel function} {in the
neighborhood of the diagonal}, see \cite{KS}.
We will discuss this in the next subsection.
\subsection{Bergman kernel}\label{SS:berg}
%
Recall from \refS{globToe} the definition of 
the Szeg\"o projectors 
$ \Pi:\Lqv\to \Hc$ and its components 
$\pimh:\Lqv\to\Hm$, the Bergman projectors.
The Bergman projectors have smooth integral kernels,
the {\it Bergman kernels} 
$\Bm(\alpha,\beta)$ defined on $Q\times Q$, i.e.
\begin{equation}
\pimh(\psi)(\alpha)=\int_Q\Bm(\alpha,\beta)\psi(\beta)\mu(\beta).
\end{equation}
The Bergman kernels can be expressed with the help of the
coherent vectors.
\begin{proposition}\label{P:kernel}
\begin{equation}
\Bm(\alpha,\beta)=\psi_{\ebm}(\alpha)=
\overline{\psi_{\eam}(\beta)}=\skp{\eam}{\ebm}.
\end{equation}
\end{proposition}
\noindent
For the proofs of this and the following propositions see
\cite{KS}, or \cite{Schlbol}.

Let $x,y\in M$ and choose $\alpha,\beta\in Q$ with
$\tau(\alpha)=x$ and $\tau(\beta)=y$ then the functions
\begin{equation}\label{E:um}
u_m(x):=\Bm(\alpha,\alpha)=
\skp{\eam}{\eam},
\end{equation}
\begin{equation}\label{E:vm}
v_m(x,y):=\Bm(\alpha,\beta)\cdot \Bm(\beta,\alpha)=
 \skp{\eam}{\ebm}\cdot \skp{\ebm}{\eam}
\end{equation}
are well-defined on $M$ and on $M\times M$ respectively.
The following proposition gives an integral representation of the
Berezin transform.
\begin{proposition}\label{P:kernelint}
\begin{equation}\label{E:kernelint}
\begin{aligned}
\left(\Imb(f)\right)(x)&=\frac 1{\Bm(\alpha,\alpha)}
\int_Q \Bm(\alpha,\beta)\Bm(\beta,\alpha)\tau^*f(\beta)\mu(\beta)
\\
&=
\frac 1{u_m(x)}
\int_M v_m(x,y)f(y)\Omega(y)\ .
\end{aligned}
\end{equation}
\end{proposition}
Typically, asymptotic expansions can be obtained using 
stationary phase integrals. But for such an asymptotic expansion
of the integral representation of the Berezin transform we
will not only need an asymptotic expansion of the Bergman kernel
along the diagonal (which is well-known) but in a
neighborhood of it.
This is one of the key results obtained in \cite{KS}.
It is based on works of Boutet de Monvel and Sj\"ostrand 
\cite{BS} on the Szeg\"o kernel and in generalization of a result of
Zelditch \cite{Zel} on the Bergman kernel on the diagonal.
The integral representation 
is used then to prove
the existence  of the 
asymptotic expansion of the Berezin transform.
See \cite{Schlbol} for a sketch of the proof.

Having such an asymptotic expansion it still remains to
identify its terms.
As it was explained in \refSS{kara}, 
Karabegov assigns to every
formal 
deformation quantizations 
with the ``separation of variables'' property 
a  {\em formal Berezin transform} $I$. 
In \cite{KS} it is shown that 
there is an explicitely specified 
star product $\ \star\ $ (see Theorem 5.9 in \cite{KS})
with associated formal Berezin transform such that 
if we replace 
$\frac 1m$ by the formal variable $\nu$ in the asymptotic  expansion of
the 
Berezin transform $\Imb f(x)$ we obtain
$I(f)(x)$.
This will finally prove \refT{btrans}.
We will exhibit the star product $\ \star\ $ in \refSS{ident}.
\medskip

Of course, for certain restricted but important non-compact
cases the Berezin transform was already introduced and 
calculated by Berezin. It was a basic tool in his approach to
quantization \cite{Berbt}.
For other  types of non-compact manifolds similar 
results on the asymptotic expansion of the Berezin transform 
are also known.
See the extensive work of Engli\v s, e.g. \cite{Eng1}.

\begin{remark}
More recently, direct approaches to the asymptotic expansion of
the Bergman kernel (outside the diagonal) 
were given, some of them yielding  
low order coefficients of the expansion.
As examples, let me mention
Berman, Berndtsson, and Sj\"ostrand, \cite{BBS}, Ma and
Marinescu \cite{MaMarBook}, Dai. Lui, and Ma \cite{DLM}. See also 
Engli\v s  \cite{Eng0}.
\end{remark}

\subsection{Proof of norm property of Toeplitz operators}
\label{SS:norm}
%
In \cite{Schlbia98} I conjectured \refE{btransas} 
(which we later proved in  joint work with Karabegov)
and showed how such an asymptotic expansion supplies a different
proof  of \refT{bms}, Part (a).
For completeness I  reproduce the proof here.
\begin{proposition}
\begin{equation}\label{E:symine}
|\Imb(f)|_\infty=|\sm(\Tfm)|_\infty\quad\le \quad||\Tfm||\quad\le\quad   
|f|_\infty\ .
\end{equation}
\end{proposition}
\begin{proof}
Using Cauchy-Schwarz inequality we calculate ($x=\tau(\alpha)$)
\begin{equation}
| \sm(\Tfm)(x)|^2=
\frac {|\skp {\eam}{\Tfm\eam}|^2}{{\skp {\eam}{\eam}}^2}\le
\frac {\skp {\Tfm\eam}{\Tfm\eam}}{\skp {\eam}{\eam}}\le
||\Tfm||^2\ .
\end{equation}
Here the last inequality  follows from the definition of the operator norm.
This shows the first inequality in \refE{symine}.
For the second inequality introduce the multiplication
operator $M_f^{(m)}$ on $\gulm$. Then 
$\ ||\Tfm||=||\Pim\,M_f^{(m)}\,\Pim||\le ||M_f^{(m)}||\ $ and
for  $\varphi\in\gulm$,  $\varphi\ne 0$
\begin{equation}
\frac {{||M_f^{(m)} \varphi||}^2}{||\varphi||^2}=
\frac {\int_M h^{(m)}(f \varphi,f \varphi)\Omega}
 {\int_M h^{(m)}(\varphi,\varphi)\Omega}
=
\frac {\int_M f(z)\overline{f(z)}h^{(m)}(\varphi,\varphi)\Omega}
{\int_M h^{(m)}(\varphi,\varphi)\Omega}
\le
|f|{}_\infty^2\ .  
\end{equation}
Hence,
\begin{equation}
||\Tfm||\le ||M_f^{(m)}||=\sup_{\varphi\ne 0}
\frac {||M_f^{(m)}\varphi||}{||\varphi||}\le |f|_\infty .
\end{equation}
\end{proof}
\begin{proof} (\refT{bms} Part (a).)
Choose as $x_e\in M$ a point with $|f(x_e)|=|f|_\infty$.
From the fact that the  Berezin
transform has
as leading term the identity   it follows that
$\ |(I^{(m)}f)(x_e)-f(x_e)|\le C/m\ $ with a suitable constant
$C$.
Hence,
$\ \left| |f(x_e)|-|(I^{(m)}f)(x_e)| \right| \le C/m\ $
and 
\begin{equation}\label{E:absch}
|f|_\infty-\frac Cm=|f(x_e)|-\frac Cm\quad\le\quad
|(I^{(m)}f)(x_e)|\quad\le\quad |I^{(m)}f|_\infty\ .
\end{equation}
Putting \refE{symine} and \refE{absch} together we obtain
\begin{equation}\label{E:thma}
|f|_\infty-\frac Cm\quad\le\quad ||T_f^{(m)}||\quad\le\quad
|f|_\infty\ .
\end{equation}
\end{proof}
\section{Berezin transform and star products}\label{S:bers}
\subsection{Identification of the BT star product}\label{SS:ident}
%
In \cite{KS} there was another object introduced, the
{\it twisted product}
\begin{equation}
R^{(m)}(f,g):=\sm(\Tfm\cdot \Tgm) \ .
\end{equation}
Also for it the 
existence of a complete asymptotic expansion was shown.
It was identified with a twisted formal product. 
This  allowed relating  the BT star product with 
a special star product  within the classification of
Karabegov. From this  the properties
of \refT{star} of
locality, separation of variables type, and the calculation to
the classifying forms and classes for the BT star product
follows.

As already announced in \refSS{kara},
the BT star product $\star_{BT}$ is the opposite of the dual star product
of a certain star product $\star$.
To identify $\star$ we will give its 
classifying Karabegov form $\widehat{\w}$ .
As already mentioned above, Zelditch \cite{Zel} proved that
the function $u_m$ \refE{um} has a complete asymptotic expansion 
in powers of $1/m$. 
In detail he showed
\begin{equation}
u_m(x)\sim m^n\sum_{k=0}^\infty \frac 1{m^k}\,b_k(x), \quad
b_0=1.
\end{equation}
If we replace in the expansion 
$1/m$ by the formal variable $\nu$ we obtain a formal function
$s$
defined by
\begin{equation}
e^{s}(x)=
\sum_{k=0}^\infty \nu^k\,b_k(x).
\end{equation}
 Now take as formal potential  \refE{formp}
$$\widehat{\Phi}=\frac 1\nu \Phi_{-1}+s,
$$  
where $\Phi_{-1}$ is the local K\"ahler potential of the 
K\"ahler form $\w=\w_{-1}$.
Then $\widehat{\w}=\mathrm{i}\;
\partial\bar\partial\widehat{\Phi}$.
It might also be written in the form
\begin{equation}\label{E:fk}
\widehat{\w}=\frac 1\nu \w+\F
(\mathrm{i}\;\partial\bar\partial\log\Bm(\al,\al)).
\end{equation}
Here we denote the replacement of $1/m$ by 
the formal variable $\nu$ by the symbol $\F$.
\subsection{The Berezin star products for arbitrary K\"ahler
  manifolds}\label{SS:btstarg}
%
We will  introduce for general quantizable compact
K\"ahler manifolds the Berezin star product.
We extract  from the asymptotic expansion of the Berezin transform 
\refE{btransas}
the {formal expression}
\begin{equation}\label{E:btform}
I=\sum_{i=0}^{\infty} I_i\;\nu^i, \quad  I_i:\cim\to\cim,
\end{equation}
as a {\it formal Berezin transform},
and set
\begin{equation}\label{E:dbs}
f\star_{B} g:=I(I^{-1}(f)\star_{BT}I^{-1}(g)).
\end{equation}
As $I_0=id$ this  {$\star_B$} is a star product
for our K\"ahler manifold, which we call the  
{\it Berezin star product}.
Obviously, the formal map 
$I$ gives the {equivalence transformation } to {$\star_{BT}$}.
Hence, the Deligne-Fedosov classes will be the same.
It will be of separation of variables type but with the
role of the variables switched.
We showed in \cite{KS} that $I=I_\star$ with star product given by
the form \refE{fk}.
We can rewrite \refE{dbs} as
\begin{equation}\label{E:btb}
f\star_{BT} g:=I^{-1}(I(f)\star_{B}I(g)).
\end{equation}
and get exactly the relation \refE{bteq}.
Hence,  $\star=\star_B$ and both star products $\star_B$ 
and $\star_{BT}$ are dual and opposite to each other.

When the definition with the covariant symbol {works} 
(explained in \refSS{borg}) $\star_B$ will
coincide with the star product defined there.

\subsection{Summary of 
naturally defined star products for compact
K\"ahler manifolds}\label{SS:summary}
%
By the presented techniques we obtained
for quantizable compact K\"ahler manifolds three different
naturally defined star products 
$\star_{BT}$,
$\star_{GQ}$, and
$\star_{B}$.
All three are equivalent and have classifying Deligne-Fedosov
class
\begin{equation}\label{E:ksepch}
cl(\star_{BT})=
cl(\star_{B})=
cl(\star_{GQ})=
\frac {1}{\mathrm{i}}(\frac {1}{\nu}[\omega]-\frac {\delta}{2}).
\end{equation}
But all three are distinct. In fact $\star_{BT}$ is of
separation of variables type (Wick-type), 
 $\star_{B}$ is of
separation of variables type with the role of the
variables switched (anti-Wick-type),
and $\star_{GQ}$ neither.
For their Karabegov forms we obtain
(see \cite{KS},\cite{SchlBer})
\begin{equation}\label{E:kf}
kf(\star_{BT})=
\frac {-1}{\nu}\omega+\omega_{can}.
\qquad
kf(\star_{B})=\frac 1{\nu}{\w}+\F(\mathrm{i}\,\dpa\db\log u_m).
\end{equation}
The function $u_m$ was introduced above as the
Bergman kernel evaluated along the diagonal in $Q\times Q$.
\begin{remark}
Based on Fedosov's method Bordemann and Waldmann \cite{BW} constructed 
also a unique star product $\star_{BW}$ which is of Wick type, see \refSS{bw}.
The opposite star product has Karabegov form 
$kf(\star_{BW}^{opp})=-(1/\nu)\,\w$ and it has Deligne-Fedosov class
$cl(\star_{BW})=\frac {1}{\mathrm{i}}(\frac {1}{\nu}[\omega]+\frac
{\delta}{2})$  \cite{Karacmf}.
It will be equivalent to $\star_{BT}$ if the canonical class is trivial.

More precisely, in \cite{Karacmf} 
Karabegov considered the ``anti-Wick'' variant
of the Bordemann-Waldmann construction. This yields a star product
with
separation of variables in the convention of Karabegov. It has 
Karabegov form $(1/\nu)\w$ and the same 
Deligne-Fedosov class as  \refE{ksepch}. Hence, it is equvialent to
$\star_{BT}$.
Recently,  in 
\cite{Kar3}, \cite{Kar4}  
Karabegov  gave a more direct construction of
the star product with Karabegov form $(1/\nu)\w$. 
Karabegov calls this star product {\it standard star product}.
\end{remark}
\subsection{Application: Calculation of the coefficients of the star
  products}\label{SS:calc}
%
The proof of \refT{star} gives a recursive definition of the
coefficients
$C_k(f,g)$. Unfortunately, it is not very  constructive. 
For their calculation the Berezin transform will also be of help.
\refT{btrans} shows for quantizable compact K\"ahler manifolds the
existence
of the asymptotic expansion of the Berezin transform \refE{btransas}.
We get the formal Berezin transform $I=\F(I^{(m)})$,
see \refE{btform}, which is the formal Berezin transform of the star
product
$\star_B$
\begin{equation*}
I=\sum_{i=0}^{\infty} I_i\;\nu^i, \quad  I_i:\cim\to\cim.
\end{equation*}
We will show that if 
we know  $I$ explicitely we obtain explicitly $\star_B$ by giving the
coefficients $C_k^{B}(f,g)$ of $\star_B$.
For this the knowledge of the coefficients
 $C_k^{BT}(f,g)$ for  $\star_{BT}$ will not be needed. All 
we need is the existence of $\star_{BT}$ to define $\star_{B}$.
The operators $I_i$ can be expressed (at least in principle) by the
asymptotic expansion of expressions formulated in terms of
 the Bergman kernel.

As $I$ is the formal Berezin transform in 
the sense of Karabegov assigned to  $\star_B$
we get for  local functions $f,g$ , $f$ anti-holomorphic, $g$ holomorphic
\begin{equation}\label{E:formd1}
f\star g=I(g\cdot f)=I(g\star f).
\end{equation}
Expanding the formal series
 for $\star_B$ \refE{stara}
and for 
$I$ \refE{btform} we get for the coefficients
\begin{equation}
C_k^B(f,g)=I_k(g\cdot f).
\end{equation}
Let us take local complex coordinates. As $\star_B$ is a differential
star product, the $C_k^B$ are bidifferential operators. As $\star_B$
is of separation of variables type, in  $C_k^B$ 
the first argument is is only 
differentiated with respect to anti-holomorphic coordinates, the second
with respect to holomorphic coordinates.
Moreover, it was shown by Karabegov
that the $C_k$ are
bidifferential
operators of order $(0,k)$ in 
the first argument and order $(k,0)$ in the second argument
and that $I_k$ is a
differential
operator of type $(k,k)$.

As $f$ is anti-holomorphic, in $I_k$ it will only see the anti-holomorphic
derivatives. The corresponding  is true for the holomorphic $g$.
By locality it is enough to consider the local functions
$z_i$ and $\zb_i$ and we get that 
$C_k^B$ can be obtained by ``polarizing'' $I_k$.

In detail,
if we write $I_k$ as summation over multi-indices 
$(i)$ and $(j)$  we get
\begin{equation}\label{E:rec1}
I_k=\sum_{(i),(j)}a_{(i),(j)}^{k}
\frac {\partial^{(i)+(j)}}
{\partial z_{(i)}\partial \zb_{(j)}},\quad
a_{(i),(j)}^{k}\in\Cim
\end{equation}
and obtain for the coefficient in the star product $\star_B$
\begin{equation}\label{E:rec2}
C_k^B(f,g)=\sum_{(i),(j)}a_{(i),(j)}^{k}
\frac {\partial^{(j)}f}
{\partial  \zb_{(j)}}
\,\frac {\partial^{(i)}g}
{\partial z_{(i)}},
\end{equation}
where the summation is limited by the order condition.
Hence, knowing the components $I_k$ of the formal
Berezin transform $I$ gives us 
$C_k^B$. 

{}From $I$ we can recursively calculate the coefficients
of the inverse $I^{-1}$ as $I$ starts with $id$. 
{}From $f\star_{BT}g=I^{-1}(I(f)\star_BI(g))$, which is  the Relation \refE{dbs} 
inverted,
we can calculate 
(at least recursively) the coefficients $C_k^{BT}$.
In practice, the recursive calculations turned out to become
quite involved.

\medskip
The chain of arguments presented above was based on
the existence of the Berezin transform and its asymptotic expansion 
for every 
quantizable compact K\"ahler manifold.
The asymptotic expansion of the  Berezin transform itself is again 
based on the  asymptotic off-diagonal expansion of the Bergman
kernel. Indeed, the Toeplitz operator can also be expressed
via the Bergman kernel.
Based on this it is clear that the same procedure will work
for  those non-compact manifolds for which  we have 
at least the same (suitably adapted) objects and corresponding
results. 
\begin{remark}
In the purely formal star product setting studied by
Karabegov \cite{Kar1} the set of star products of separation of variables type,
the set of formal Berezin transforms, and the 
set of formal Karabegov forms are in 1:1 correspondence.
Given $I_\star$ the star product $\star$ can be recovered via the 
correspondence \refE{rec1} with \refE{rec2}.
What generalizes  $\star_{BT}$ in this more general setting
is the dual and opposite of
$\star$.
\end{remark}

\begin{example}
As a  simple 
but nevertheless instructive case let us consider $k=1$.
Recall that $n$ is the complex dimension of $M$. 
Starting from our K\"ahler form $\omega$ expressed in local
holomorphic coordinates $z_i$ as 
$\omega=\iu\sum_{i,j=1}^{n}g_{ij}dz_i\wedge d\zb_j$ the
Laplace-Beltrami operator is given by
\begin{equation}\label{E:delta}
\Delta=\sum_{i,j}g^{ij}\frac {\partial^2}{\partial z_i\partial
  \zb_j},
\end{equation}
here $(g^{ij})$ is as usual the inverse matrix to $(g_{ij})$.
\footnote{From the context it should be clear  that $g$ and
  $g_{ij}$ are unrelated objects.}
The Poisson bracket is given  (up to  $\epsilon$ which is 
a factor of signs, 
complex units, and factors of 1/2 due to preferred conventions) by
\begin{equation}\label{E:pois}
\{f,g\}=\epsilon\cdot\sum_{i,j}g^{ij}
\left(\frac {\partial f}{\partial \zb_i}
 \frac{\partial g}{\partial z_j} 
-
\frac{\partial f}{\partial z_j}
\frac{\partial g}{\partial \zb_i}
\right)
\end{equation}
From $I_1=\Delta$ we deduce immediately with \refE{delta}
\begin{equation}\label{E:c1b}
C_1^B(f,g)=
\sum_{i,j}g^{ij}
\frac {\partial f}{\partial \zb_i}
 \frac{\partial g}{\partial z_j}.
\end{equation}
The inverse of $I$ starts with $id-\Delta \nu+....$.
If we isolate using \refE{btb} from 
\begin{equation}
(id-\Delta \nu)(((id+\Delta \nu)f)\star_B
((id+\Delta \nu)g))
\end{equation}
the terms of order one in $\nu$ we get
\begin{equation}\label{E:c1bt}
C_1^{BT}(f,g)=
C_1^B(f,g)+(\Delta f)g+f(\Delta g)-\Delta (fg)
=
-\sum_{i,j}g^{ij}
\frac {\partial f}{\partial z_i}
 \frac{\partial g}{\partial \zb_j}.
\end{equation}
This is of course not  a surprise.
We could have it deduced also directly. Our star products are
of separation of variables type and the $C_1$ have to have a form
like \refE{c1b} (or \refE{c1bt}) with coefficients $a^{ij}$ which
a priori could be different from $g^{ij}$ and $-g^{ij}$ respectively.
From $C_1(f,g)-C_1(g,f)=-\iu\{f,g\}$ 
it follows that they are equal.
\end{example} 
Calculating  the higher orders can become quite tedious.
First of course the Berezin transform is only known in closed
form for certain homogeneous spaces. For general (compact) manifolds
by \refP{kernelint}
its asymptotic expansion can be expressed in terms of
asymptotic expansions of the Bergman kernel. The 
Bergman kernel can be expressed locally with respect to 
adapted coordinates via data associated to the K\"ahler metric.
Hence the coefficients $C_k^B$ and $C_k^{BT}$ can be also
expressed in these data. 
In case that the Berezin transform exists it was an important
achievement of Mirek Engli{\v s} to exploit this in detail
also in the noncompact case, under the condition that the
Berezin transform exists
\cite{Eng0}, \cite{Engbk}. He calculated small 
order terms in the star products.

Later, Marinescu and Ma 
used also  Bergman kernel techniques in a different way 
even in the 
case of compact symplectic manifolds and orbifolds
and allowing an auxiliary vector  bundles.
In their approach they introduced Toeplitz kernels 
and calculated small order terms for the Berezin-Toeplitz star
product \cite{MaMar}. A Berezin transform does not show up.
See \cite{MaMarRev} for a review of their techniques, results
and further reference to related literature.
See also results of Charles 
\cite{Chardiss}, \cite{Char},  \cite{Charhf}, \cite{Charl}.

\section{Other constructions of star products -- Graphs}\label{S:graph}

\subsection{Bordemann and Waldmann}\label{SS:bw}
\cite{BW}
Fedosov's proof of the existence of a star product for every
symplectic
manifold was geometric in its very nature \cite{Fed}.
He considers a certain infinite-dimensional bundle $\hat{W}\to
M$ of formal series of symmetric and antisymmetric forms on the
tangent bundle of $M$. For this bundle he defines the fiber-wise 
Weyl product.
Denote by $\hat{\mathcal W}$ the sheaf of smooth sections of this
bundle, with $\circ$ as induced product.

Starting from a symplectic torsion free
connection
he constructs recursively what is called the Fedosov derivation  
$D$ for
the sheaf of sections. It is flat, in the sense that $D^2=0$.
 The kernel of $D$ is a $\circ$-subalgebra. Let
${\mathcal W}$ be the elements of $\hat{\mathcal W}$ 
for which the values have antisymmetric
degree
zero. 
The natural projection to the symmetric degree zero part gives 
a linear isomorphism 
  from the $\circ$-subalgebra 
$\sigma:{\mathcal W}_D=\ker D\cap {\mathcal W}
\to \Cim[[\nu]]$.
The algebra structure of  ${\mathcal W}_D$ gives the star product we
were looking for, i.e. 
$f\star g:=\sigma(\tau(f)\circ\tau(g))$ with $\tau$ the inverse of
$\sigma$ which  recursively can  by calculated.

In case that $M$ is an arbitrary K\"ahler manifold, Bordemann and
Waldmann \cite{BW} were able to modify the set-up by taking the
fiber-wise Wick product.
By a modified Fedosov connection a star
product $\star_{BW}$  is obtained 
which is of Wick type, i.e. $C_k(.,.)$ for $k\ge 1$ has only 
holomorphic derivatives in the first argument and anti-holomorphic
arguments in the second argument.
Equivalently, it is of separation of variables type.
As already remarked earlier, its  Karabegov form is $-(1/\nu)\omega$
and it has Deligne-Fedosov class
$cl(\star_{BW})=\frac {1}{\mathrm{i}}(\frac {1}{\nu}[\omega]+\frac
{\delta}{2})$. It will be equivalent to the BT star product if the
canonical class is trivial.

Later Neumaier \cite{Neu1} was able to show that each star product
 of separation of variables type 
(i.e. the star products opposite to the  Karabegov
star product from \refSS{kara})
can be obtained by the Bordemann-Waldmann construction  
by  adding  a formal closed $(1,1)$ form as
parameter in the construction.
\subsection{Reshetikhin and Takhtajan}\label{SS:resta}
\cite{ResTak}
In the following subsections we will indicate certain relations between
the question of existence and/or
 the calculation of coefficients  of
star products and their description 
by graphs. One of the problems in the context of star products is that 
the questions reduce often to rather intricate combinatorics of 
derivatives of the involved functions and other ``internal'' 
geometrical data coming from the manifold, like Poisson form, K\"ahler
form,
etc. 
One has to keep track of multiple derivations of
many products and sums involving tensors related to the Poisson
structure,
metric, etc.  and the
functions $f$ and $g$.
In this respect graphs are usually a very helpful tool to control the
combinatorics and to find ``closed expressions'' in terms of graphs.

\medskip

Berezin in his approach  to define 
a star product  for complex domains in $\C^n$ used 
analytic integrals
depending on a real parameter $\hbar$.
Compare this to  \refE{bsint} where due to compactness 
we have a discrete parameter $1/m$.
In these integrals scalar products of coherent states
show up. Similar to \refP{kernel} they are identical to the 
Bergman kernel. Under the condition that the K\"ahler form is
real-analytic its K\"ahler potential $\Phi$ admits an
analytic continuation $\Phi(z,\wb)$ on $\C^n\times\C^n$.
\footnote{ In this subsection for the formalism of
analytic continuation, it is convenient to write $f(z,\zb)$ for a function
$f$ on $M$  to indicate 
its dependence on holomorphic and anti-holomorphic directions.}
The Bergman kernel can be rewritten 
with a suitable complementary factor $e_{\hbar}(z,\wb)$ as
\begin{equation}\label{E:bsplit}
\Bh(z,\wb)=\e^{\Phi(v,\wb)}e_{\hbar}(z,\wb).
\end{equation}
Moreover, one considers Calabi's diastatic function
\begin{equation}
\Phi(z,\zb,w,\wb)=\Phi(z,\wb)+\Phi(w,\zb)-\Phi(z,\zb)
-\Phi(w,\wb).
\end{equation}
The corresponding integral rewrites as
\begin{equation}\label{E:bint}
(f\star_{\hbar}g)(z,\zb)
=\int_{\C^n}f(z,\wb)g(w,\zb)
\frac {e_{\hbar}(z,\wb)e_{\hbar}(w,\zb)}{e_{\hbar}(z,\zb)}
e^{(\Phi(z,\zb,w,\wb)/\hbar}\,\Omega_\hbar,
\end{equation}
where $\Omega_\hbar$ is the $\hbar$ normalized 
Liouville form.
To show that the 
integral gives  indeed a star product
Berezin needs the crucial assumption  $e_{\hbar}(z,\wb)$ is constant.
The desired results are obtained via the 
Laplace method.

Reshetikhin and Takhtajan 
consider now such type of integrals 
(still ignoring the $e_{\hbar}(z,\wb)$) as formal integrals and make
a formal Laplace expansion to obtain a ``star'' product, 
which we denote for the
moment by $\bullet$.
The coefficients of the expansion for $f\bullet g$ can be expressed with the
help of partition functions of a restricted set ${\mathcal G}$ of
locally oriented graphs (Feynman diagrams)
fulfilling some additional conditions and equipped with 
additional data.
In particular,  each $\Gamma\in \mathcal{G}$ contains two special
vertices, a vertex  $R$  with only incoming edges and
and a vertex $L$ with only outgoing edges.
Furthermore, the other vertices are divided into 
two sets, the solid and the hollow vertices.
The ``star'' product for $\C^n$ as formal power series in $\nu$
can be written as
\begin{equation}\label{E:rt}
f\bullet g=
\sum_{\Gamma\in {\mathcal G}}\frac {\nu^{\chi(\Gamma)}}{|Aut(\Gamma)|}
D_\Gamma(f,g).
\end{equation}
Here 
 $Aut(\Gamma)$ is the
subgroup of automorphism of the graph $\Gamma$ respecting the special
structure, $\chi(\Gamma)$ is the number of edges of $\Gamma$ minus
the number of ``solid'' vertices.
The crucial part is 
$D_\Gamma(f,g)$  {\it the partition function}
of the graph $\Gamma$ equipped with certain additional
data. It encodes the information from the formal
expansion of the integral associated to this graph.
The special vertex $L$ is responsible for
differentiating $f$ with respect to anti-holomorphic coordinates
and $R$ for
differentiating $g$ with respect to holomorphic coordinates.
It is sketched that 
the product $\bullet$ is ``functorial''
with respect to holomorphic changes of coordinates
and that it defines a formal deformation quantization
for any  arbitrary complex manifold $M$ with K\"ahler form $\omega$.
But as 
in general $1\bullet f\ne f\ne f\bullet 1$, i.e. it is
not null on constants.
Essentially this 
is due to the fact, that the complementary factors
$e_{\hbar}(z,\wb)$ \refE{bsplit}
were not taken into account.
But the obtained algebra contains a unit element 
$e_{\nu}(z,\zb)$ which is
invertible.
This unit is used to twist $\bullet$ 
\begin{equation}
(f\star g)(z,\zb)=e_{\nu}^{-1}(z,\zb)
((f\cdot e_\nu)\bullet (g\cdot e_\nu))
\end{equation}
to obtain a star
product $\star$  which is null on constants.
As the notation already indicates, 
the unit $e_{\nu}(z,\zb)$ is related to the formal
Bergman kernel evaluated along the diagonal.
\subsection{Gammelgaard}\label{SS:gam}\cite{Gam}
His starting point is the formal deformation $\widehat{\w}$ of the
pseudo-K\"ahler form $\w=\w_{-1}$ given by \refE{defw}.
Let $\star$ be the unique star product of separation of variables 
type (in the convention of Karabegov) 
associated to $\widehat{\w}$ which exists globally.
Gammelgaard gives a local expression of this star product by 
\begin{equation}\label{E:gam}
f\star g=
\sum_{\Gamma\in{\mathcal A}_2}
\frac {\nu^{W(\Gamma)}}{|Aut(\Gamma)|}
D_{\Gamma}(f,g).
\end{equation}
This looks similar to \refE{rt} but of the set of graphs to
be considered are different. Also the partition functions
will be different.
Local means that he chooses for every point a contractible
neighborhood such that $\widehat{\w}$ has a formal potential
\refE{formp}.
The set ${\mathcal A}_2$ is the set of isomorphy classes 
of directed acyclic graphs
(parallel edges are allowed) which have exactly one vertex which is
a sink (i.e. has only  incoming edges) and  one vertex which is
a source (i.e. has only outgoing edges). These two vertices are called
external vertices, the other internal.
As usual we denote by $E$ the set of edges and by $V$ the set of
vertices of the graph $\Gamma$. The graphs are weighted by assigning
to every internal vertex $v$ an integer $w(v)\ge -1$.
Each internal vertex has at least one incoming and one 
outgoing edge.
If  $w(v)=-1$ then  at least 3 edges are connected with $v$.
The total weight $W(\Gamma)$ of the graph $\Gamma$ 
is defined as the sum $W(\Gamma):=|E|+\sum_{v \ \text{internal}}
w(v)$.
Isomorphism are required to respect the structure.
Also in this sense $|Aut(\Gamma)|$ has to be understood.

To each such graph a certain bidifferential operator is assigned. It
involves the geometric data and the functions $f$ and $g$. 
The function $f$ corresponds to the external vertex which is a 
source and   $g$ to the sink.
The internal vertices of weight $k$ involve  $-\Phi_k$ from
\refE{formp}.
Incoming edges correspond to taking derivatives with
respect to holomorphic coordinates, outgoing with respect to
anti-holomorphic coordinates.
Hence $f$ is only differentiated with respect to 
anti-holomorphic and $g$ with respect to holomorphic.
The partition function is now obtained by contracting the tensors
with the help of the K\"ahler metric.

In the main part of the paper \cite{Gam} 
Gammelgaard shows that this definition 
is indeed associative and defines locally a star product  with 
the (global) Karabegov form $\widehat{\w}$ he started with.
Hence it is the local restriction of $\star$.

The formula is particularly nice if there are not so many 
higher order  terms
in $\widehat{\w}$. 
For example for $\widehat{\w}=(1/\nu)\w_{-1}$, i.e. the
``standard star product'' only
those graphs contribute for which all vertices have weight $-1$.
For the Berezin star product we will have in general higher degree
contributions, see \refE{kf}.
But the opposite of the Berezin-Toeplitz 
star product has Karabegov form 
$-(1/\nu)\omega+\w_{can}$, hence only graphs which have only
vertices of weight $-1$ or $0$ will contribute.
As Gammelgaard remarks this allows to give explicit formulas for 
the coefficients of the BT star product. Recall that for the
opposite star product only the role of $f$ and $g$ is switched.

As an example let me derive the ``trivial coefficients''.
The only graph of weight zero is the one consisting on the
two external vertices and no edge. Hence $C_0(f,g)=f\cdot g$ as
required. The only graph of weight one consists of
the two external vertices and a directed edge between them. 
Hence, we obtain for every $\widehat{\w}=(1/\nu)\omega_{-1}+...$
the expression \refE{c1b}, and for the Berezin-Toeplitz star product
\refE{c1bt} (note that we have to take the pseudo-K\"ahler form
$-\omega_{-1}$ and switch the role of $f$ and $g$).
Internal vertices will only show up for weights $\ge 2$.

\subsection{Huo Xu}\label{SS:xu}\cite{Huo1},\cite{Huo2}
His starting point is the Berezin transform.
Let us assume it exists, which  
at least is true in the case of compact quantizable K\"ahler
manifolds.
As explained in \refSS{calc} via the formula \refE{rec2}  the
coefficients
of the Berezin star product are given.
Based on Engli{\v s}'s work \cite{Eng0} Huo Xu found 
a very nice
way to deal with the Bergman kernel \cite{Huo1}
in terms of
certain graphs.
In \cite{Huo2}
he  applies the result to the Berezin transform and 
Berezin star product.
His formula for the product is
\begin{equation}
f\star_B g=\sum_{\Gamma\in\mathcal{G}}
\frac {\det(A(\Gamma_-)-Id)}{|\mathrm{Aut}'(\Gamma)|}
 \nu ^{|E|-|V|}\,
D_\Gamma(f,g)=\sum_{k=0}^{\infty}C_k^B(f,g)\nu^k.
\end{equation}
Here $\mathcal{G}$ is a certain subset of pointed directed graphs 
(i.e. in technical terms 
it is the set of strongly connected pointed stable graphs
-- loops and cycles are allowed)
consisting of the vertices $V\cup {v}$ (with $v$ the distinguished 
vertex) and edges $E$. After erasing the vertex $v$ the graph
$\Gamma_-$ is obtained. Now $A(\Gamma_-)$ is its adjacency matrix.
$|\mathrm{Aut}'(\Gamma)|$ is the number of automorphisms 
of the pointed graph  fixing the distinguished vector. 
The only object which is a function is again the 
{\it partition function} $ D_\Gamma(f,g)$ of the graph defined like 
follows. Each such graph $\Gamma$ encodes a ``Weyl invariant'' given
in
terms of partial derivatives and contractions of the metric. 
This defines the partition function, whereas the distinguished 
vertex is replaced by ``$f$'' and ``$g$''. All incoming edges are
associated to $f$ and correspond to $\frac {\partial}{\partial \zb_i}$
derivatives and all outgoing are associated to $g$ and 
correspond to $\frac {\partial}{\partial z_i}$. 
For the precise formulations of his results I refer to his work.

For small orders he classifies the graphs and calculates 
for $k$ up to three the $C_k^B(f,g)$ and   $C_k^{BT}(f,g)$ in terms of
the metric data.
But again the reformulation to explicit formulas tend to become 
quite involved with increasing $k$.

\medskip

The approaches via graphs presented in Sections
\ref{SS:resta}.,\ref{SS:gam}, and \ref{SS:xu} for sure are in some sense
related as they center around the same objects. 
But the set of graphs considered are completely different.
Further investigation is necessary to understand this relation.
See in this direction the very recent preprint of Xu \cite{Huo3}.


\section{Excursion: The Kontsevich construction}\label{SS:kon}

Kontsevich showed in \cite{Kont} the existence of a star product for
every
Poisson manifold $(M,\{.,.\})$.
In fact he proves the more general formality conjecture which implies
the
existence. It is not my intention even to give a sketch of this here.
Furthermore, in the K\"ahler case we are in the symplectic
case and there are other existence and classification proofs obtained
much earlier.
Nevertheless, as we are dealing with graphs and star product in 
the previous section, it is very interesting  
to sketch his explicit formula for the
star product in terms of Feynman diagrams.

He considers 
star products for  open sets in $\mathbb{R}^d$ with
arbitrary Poisson structure given by 
the Poisson bivector
$\alpha=(\alpha^{ij})$. In local coordinates $\{x_i\}$ the
Poisson bracket is given as
\begin{equation}
\{f,g\}(x)=\sum_{i,j=1}^d\alpha^{ij}(x)
 {\partial_i f} {\partial_j g},
\quad \partial_i:=\frac {\partial}{\partial x_i}.
\end{equation}
The star product is defined by
\begin{equation}\label{E:skont}
f\star g=f\cdot g+\sum_{n=1}^\infty
\left(\frac {\mathrm{i}\,\nu}{2}\right)^{n}
\sum_{\Gamma\in \mathcal{G}_n}w_\Gamma D_\Gamma(f,g).
\end{equation}
Here $\mathcal{G}_n$ is a certain subset of graphs of order $n$,
and the partition function
$D_\Gamma$ is a bidifferential operator involving the 
Poisson bivector $\alpha$ (of homogeneity $n$).
The graph $\Gamma$ 
encodes which derivatives have to be taken in $D_\Gamma$ and
$w_\Gamma$ is a weight function.

More precisely,
$\mathcal{G}_n$ consists of oriented graphs with $n+2$ vertices, labeled by
$1,2,\ldots,n,L,R$, such that at each numbered vertex $[i]$, 
$i=1,\ldots,n$ exactly
two edges $e_i^1=(i,v_1(i))$ and $e_i^2=(i,v_2(i))$  start
and end at two different other vertices (including $L$ and $R$) but not
at $[i]$ itself. Each such graphs has $2n$ edges. 
Denote by $E_\Gamma$ the set of edges.
The number of graphs in $G_n$ is $(n(n+1))^2$ for $n\ge 1$ and
$1$ for $n=0$.
The bidifferential operator is defined by
\begin{equation}
\begin{split}
D_\Gamma(f,g):=\sum_{I:E_\Gamma\to \{1,2,\ldots,d\}}
\left(
\prod_{k=1}^n\left(\prod_{\substack{ e\in E_\Gamma\\
            e=(*,k)}} \partial_{I(e)}\right)
\alpha^{I(e_k^1)I(e_k^2)}\right)\times
\\
\times\left( \prod_{\substack{e\in E_\Gamma\\
            e=(*,L)}} \partial_{I(e)}\right)f\cdot
\left( \prod_{\substack{e\in E_\Gamma\\
            e=(*,R)}} \partial_{I(e)}\right)g.
      \end{split}
    \end{equation}
The summation can be considered as assigning to the $2n$ edges 
independent indices $1\le i_1,i_2,\ldots, i_{2n}\le d$ as labels.
\begin{example}
Let $\Gamma$ be the graph with vertices $(1,2,L,R)$ and edges
$$
e_1^1=(1,2),\quad e_1^2=(1,L),\quad e_2^1=(2,L),\quad  e_2^2=(2,R).
$$
Then 
$$
D_\Gamma(f,g)=\sum_{i_1,i_2,i_3,i_4=1}^d
(\alpha^{i_1i_2})(\partial_{i_1}\alpha^{i_3i_4})
(\partial_{i_2}\partial_{i_3}f)(\partial_{i_4}g).
$$
\end{example}
The weights $w(\Gamma)$ are calculated by considering the upper half-plane
$H:=\{z\in\mathbb{C}\mid Im(z)>0\}$ with the Poincare metric.
Let $C_n(H):=\{u\in H^n\mid u_i\ne u_j, \text{ for } i\ne j\}$ be the
configuration space of $n$ ordered distinct points on $H$.
For any two points $z$ and $w$ on $H$ we denote by 
$\phi(z,w)$ the (counterclock-wise) 
angle between the geodesic connecting $z$ and
$\mathrm{i}\infty$
(which is a straight line) and the geodesic between $z$ and $w$.
Let  $d\phi(z,w)=\frac {\partial}{\partial z}\phi(z,w)dz+
 \frac {\partial}{\partial w}\phi(z,w)dw$ be the differential.
The weight is then defined as
\begin{equation}
w_\Gamma=\frac {1}{(2\pi)^{2n}n!}\int_{C_n(H)}
\wedge_{i=1}^n d\phi(u_i,u_{v_1(i)})\wedge d\phi(u_i,u_{v_2(i)}),
\end{equation}
with the convention that for 
$L$ and $R$ the values at the boundary (of $H$)
$u_L=0$ and $u_R=1$ are taken.

\begin{remark}

In \cite{CF} Cattaneo and Felder 
gave a field-theoretical interpretation of
the formula \refE{skont}.
They introduce a sigma model defined on the unit disc $D$
(conformally equivalent to the upper half-plane)
with values in the Poisson manifold $M$ as target space.
 The model contains two 
bosonic fields: (1) $X$, which is function on the disc, and (2) $\eta$, 
which is a differential 1-form on $D$ taking values in the pullback 
under  $X$ of the cotangent bundle of $M$,
i.e. a section of $X^*(T^*M)\otimes T^*D$.

In local coordinates $X$ is given by $d$ functions $X_i(u)$ and 
$\eta$ by $d$ differential 1-forms 
$\eta_i(u)=\sum_{\mu}\eta_{i,\mu}(u)du^{\mu}$.
The boundary condition for $\eta$ is that for $u\in\partial D$, 
$\eta_i(u)$ vanishes on vectors tangent to $\partial D$.
The action  is defined as
\begin{equation}
S[X,\eta]=\int_D\sum_i\eta_i(u)\wedge dX^i(u)+\frac 12\sum_{i,j}
\alpha^{ij}(X(u))
\eta_i(u)\wedge \eta_j(u).
\end{equation}
If $0,1,\infty$ are any three cyclically ordered points on the 
boundary of the disc, the star product can be given 
(at least formally)
as the semi-classical expansion of the path-integral
\begin{equation}
f\star g\;(x)=\int_{X(\infty)=x}f(X(1))g(X(0))\exp(\frac {i}{\hbar}
S[X,\eta])dXd\eta\ .
\end{equation}
To make sense of the expansion  and to  perform the quantization
a gauge action has to be divided out. 
After this the  same formula as by Kontsevich is obtained,
except that 
in the sum over the graphs also graphs with loops (also called
tadpoles) appear.
The corresponding integrals which supply the weights associated to the
graphs with loops are not absolutely convergent. 
These graphs are removed by a certain technique called 
finite renormalization.
In this way Cattaneo and Felder
 give a very elucidating (partly heuristic) approach
to Kontsevich formula for the star product.
\end{remark}

How the Kontsevich construction is related to the other
graph construction presented in \refS{graph} is unclear at the
moment.  
\section{Some applications of the Berezin-Toeplitz operators}
\label{S:appl}

In this closing section  we will give some references indicating 
some applications 
of the Berezin-Toeplitz quantization scheme. The interested reader is
invited to check the quoted literature for full details, 
and more references. 
This list of applications and references is rather incomplete.


\subsection{Pull-back of the Fubini-Study metric, extremal metrics,
balanced embeddings}

Let  $(M,\omega)$ be a K\"ahler manifold 
with very ample quantum line bundle $L$.  After choosing an
orthonormal basis of the space $\ghm$ we 
can use them to construct an embedding 
$\phi^{(m)}:M\to\Pro^{N(m)}$ 
of $M$ into projective space of dimension $N(m)$, see \refR{ample}.
On  $\Pro^{N(m)}$ we have as standard K\"ahler form the Fubini-Study  form
$\omega_{FS}$ (and its associated metric).
The pull-back
$(\phi^{(m)})^*\omega_{FS}$ will define a K\"ahler form on $M$. It 
will not depend on the orthogonal basis
chosen for the embedding. 
In general it will not coincide with 
a scalar multiple of the K\"ahler form $\omega$ we started with
(see \cite{BerSchlcse} for a thorough discussion of the
situation).

It was shown by Zelditch \cite{Zel}, by generalizing a result
of Tian \cite{Tian} and Catlin \cite{Cat}, that  
$(\Phi^{(m)})^*\omega_{FS}$ admits a complete asymptotic expansion in 
powers of $\frac 1m$ as $m\to\infty$.

In fact it is related to the asymptotic expansion of the
Bergman kernel  \refE{um} along the diagonal. 
The pullback calculates as
\cite[Prop.9]{Zel}
\begin{equation}\label{E:zel}
\left(\phi^{(m)}\right)^*\w_{FS}=m\w+\iu\partial\bar\partial\log u_m(x)\ .
\end{equation}
In our context of star products it is interesting to note that if
in \refE{zel} 
we  replace $1/m$ by $\nu$ we obtain 
the Karabegov form of the star product $\star_B$  \refE{kf}
\begin{equation}
\widehat{\w} =\F(\left(\phi^{(m)}\right)^*\w_{FS}).
\end{equation}

The  asymptotic expansion of $(\phi^{(m)})^*\w_{FS}$
is called Tian-Yau-Zelditch expansion.
Donaldson \cite{Don1}, \cite{Don2} took it as the starting point 
to study the existence and uniqueness of
constant scalar curvature K\"ahler metrics $\w$ on  compact manifolds.
If they exists at all he 
approximates them 
by using so-called balanced metrics 
on sequences of powers of the line bundle $L$
obtained by balanced embeddings.
Balanced embeddings are embeddings fulfilling certain additional
properties introduced by Luo \cite{Luo}. They are related to stability
of the embedded manifolds in the sense of classifications in 
algebraic geometry.

It should be remarked that the ``balanced condition'' is 
equivalent to the fact that Rawnsley's \cite{Raw} epsilon function
\refE{rawe} is constant.
See also \cite[Prop.6.6]{SchlBer}.
 This function was introduced in 1975 by Rawnsley in the context of
quantization of K\"ahler manifolds and further developed by
Cahen, Gutt, and Rawnsley \cite{CGR1}. 
In particular it will be constant if the quantization is ``projectively
induced'', i.e. coming from the projective space of the coherent
state embedding \refE{cohse}. 
See \refSS{borg} for consequences about the possibility of 
Berezin's original construction of a star product.

Let me give beside the already mentioned a few more names
related to the existence and uniqueness of 
constant scalar curvature K\"ahler metrics:
Lu \cite{Lu}, Arezzo and Loi \cite{ArLoi},
Fine \cite{Fine}.
For sure much more should be mentioned, but space limitation
do not allow.


\subsection{Topological quantum field theory and 
mapping class groups}

In the context of Topological Quantum Field Theory (TQFT) the
moduli space $M$ of gauge equivalence classes of flat $SU(n)$ connections 
(possibly with monodromy around a fixed point)
over a compact Riemann surface  $\Sigma$ plays an important role.
This moduli space carries a symplectic structure $\omega$ and a
complex
line bundle $L$.
After choosing a complex structure on $\Sigma$ this moduli space will
be endowed with a complex structure, 
$\omega$ will become a K\"ahler form and 
$L$ get a holomorphic structure. Moreover $L$ will be a quantum
line bundle in the sense discussed in this review. Hence, we can employ
the Berezin-Toeplitz quantization procedure to it.
The quantum space of level $m$ will be as above the 
(finite-dimensional) space of
holomorphic sections of the bundle $L^m$ over  $M$. 
If we vary the complex structure on $\Sigma$ 
the differentiable (symplectic) data will stay the same, but 
the  
complex geometric
data will vary. In particular, our family of quantum spaces will define
a vector bundle over the Teichm\"uller space (which is the space
of complex structures on $\Sigma$).
This bundle is called the \emph{Verlinde bundle} of level $m$.
There is a canonical projectively flat connection  for this bundle, the
Axelrod-de la Pietra-Witten/Hitchin connection.

Via the
projection  to the subspace of holomorphic section, 
the Toeplitz operators will depend on the complex structure.
For a fixed differentiable function $f$ on the moduli space 
of connections the Toeplitz operators will define a section
of the endomorphism bundle of the Verlinde bundle.

The mapping class group acts on 
the geometric situation.
In particular, it acts on the space of projectively 
covariant constant sections of
the Verlinde bundle.
This yields a representation of the mapping class group. 
By general results about the order of the elements in the
mapping class group it cannot act faithfully.
But it was a conjecture of Tuarev, that 
at least it acts asymptotically faithful. This says that given a
non-trivial element
of the mapping class group there is a level $m$ such that the element
has a non-trivial action on the space of 
projectively covariant constant sections of the Verlinde bundle of 
level $m$.

This conjecture was shown by J. Andersen in a beautiful 
proof using Berezin-Toeplitz techniques. For an exact formulation of
the statement see \cite{And1}, resp. the  overview by
Andersen and Blaavand \cite{AndBla}, and
\cite{Schlsu}.

With similar techniques Andersen could show that 
the mapping class groups $\Gamma_g$ for genus $g\ge 2$ do not
have Property (T) \cite{Andpt}. Roughly speaking Property (T) means
that
the trivial representation is isolated  (with respect to
a certain topology) in the space of all unitary representations.

There are quite a number of other interesting results shown 
and techniques developed by Andersen using Berezin-Toeplitz
quantization
operators and star products, e.g. in the context of
Abelian Chern-Simons Theory \cite{And2}, 
modular functors (joint with K. Ueno) \cite{AU},
and formal Hitchin connections \cite{AnGa}.


\subsection{Spectral theory -- quantum chaos}

The large tensor power behaviour of the sections of the
quantum bundle and of the Toeplitz operators are of interest.

Shiffman and Zelditch considered in \cite{ShiZel} the 
limit distribution of zeros of such sections. The results are 
related to models in \emph{quantum chaos}. 
See also other publications of the same authors.

As mentioned in \refS{btq},
the Toeplitz operators associated to real-valued functions 
are self-adjoint. Hence, they have a real  spectrum.
With respect to this the following result 
on the trace is of importance
\begin{equation}\label{E:trop}
\Tr^{(m)}\,(T^{(m)}_{f})
        =m^n\left(\frac 1{\volu (\Pro^{n}(\C))}
\int_M f\, \Omega +O(m^{-1})\right)\ .
\end{equation}
Here $n=\dim_{\C}M$ and $\Tr^{(m)}$ denotes the trace on
$\End(\ghm)$.
See \cite{BMS}, resp. \cite{Schldef} for a detailed proof.

On the spectral analysis of Toeplitz operators see
e.g. articles by Paoletti
\cite{Pao1}, \cite{Pao2}, \cite{Pao3}.
For relation to index theory see e.g.
work of Boutet de Monvel, Leichtnam, Tang, and Weinstein 
\cite{BLTW}, and 
Bismut, Ma, and Zhang \cite{BMZ}.


\subsection{Automorphic forms}

Another field where the set-up developed in this review shows
up in a natural way is the theory  of automorphic forms.
For example, let $B^n=SU(n,1)/S(U(n)\times U(1))$ be the open unit
ball and $\Gamma$ a discrete, cocompact subgroup of $SU(n,1)$ then
the quotient $X=\Gamma/B^n$ is a compact complex manifold.
Moreover,  the invariant K\"ahler form on $B^n$ will 
descends to a K\"ahler form $\omega$ on the quotient.
The canonical line bundle (i.e. the bundle of holomorphic
$n$-forms) is a quantum line bundle for $(X,\omega)$.

By definition the sections of the tensor powers of 
this line bundle correspond to functions 
 on $B^n$ which 
are equivariant under the action of $\Gamma$ with a certain
factor of automorphy. 
In other words they are automorphic forms.
The power of the factor of automorphy  is related to the 
tensor power of the bundle.
An important problem is to construct sections, resp. automorphic
forms. For example, Poincar\'e series are obtained by an averaging
procedure and give naturally such sections. But it is not
clear that they are not identically zero.
T. Foth \cite{Foth1}  worked in the frame-work
of Berezin-Toeplitz operators to show that at least
for higher tensor powers there are non-vanishing Poincar\'e 
series.
In this process she used 
techniques proposed by 
Borthwick, Paul, and Uribe \cite{BPU}
and assigns to Legendrian 
 tori sections of
the bundles. By asymptotic expansion the non-vanishing follows.
See also  \cite{Foth2}.

\bibliographystyle{amsplain}

\end{document}